\documentclass[12pt,a4paper, fullpage]{article}
\usepackage{hyperref}
\usepackage{url}
\usepackage{bbm}
\usepackage[backend=bibtex, style=numeric,
    giveninits=true, maxbibnames=5]{biblatex}
\usepackage{amsthm}
\usepackage{amsmath}
\usepackage{amssymb}
\usepackage{enumitem}
\usepackage{color}

\usepackage{mathtools}
\usepackage[usenames,dvipsnames]{xcolor}
\usepackage{gensymb}
\usepackage{tabularx}
\usepackage{multirow}
\usepackage{aligned-overset}
\usepackage{harpoon}
\usepackage{overpic}
\usepackage{algorithm}
\usepackage{algpseudocode}

\usepackage[left=3cm, right=3cm, top=3cm, bottom=4cm]{geometry}
\usepackage{tikz}
\usetikzlibrary{patterns}
\numberwithin{equation}{section}
\allowdisplaybreaks
\definecolor{myblue}{RGB}{0, 0, 0}

\hypersetup{
    unicode=false,          % non-Latin characters in Acrobat’s bookmarks
    pdftoolbar=true,        % show Acrobat’s toolbar?
    pdfmenubar=true,        % show Acrobat’s menu?
    pdffitwindow=false,     % window fit to page when opened
    pdfstartview={FitH},    % fits the width of the page to the window
    pdftitle={},    % title
    pdfauthor={Lukas Weissinger},     % author
    pdfsubject={paper},   % subject of the document
    pdfcreator={},   % creator of the document
    pdfkeywords={}, % list of keywords
    pdfnewwindow=true,      % links in new PDF window
    colorlinks=true,       % false: boxed links; true: colored links
    linkcolor=myblue,          % color of internal links (change box color with linkbordercolor)
    citecolor=myblue,        % color of links to bibliography
    linkbordercolor={0 0 0},
    filecolor=black,      % color of file links
    urlcolor=black           % color of external links
}

\if@twoside
\else 
\fi

\numberwithin{equation}{section}
\counterwithin{figure}{section}
\counterwithin{table}{section}

\theoremstyle{plain}% default
\newtheorem{theorem}{Theorem}[section]
\newtheorem{lemma}[theorem]{Lemma}
\newtheorem{proposition}[theorem]{Proposition}
\newtheorem{corollary}[theorem]{Corollary}

\theoremstyle{definition}
\newtheorem{definition}{Definition}[section]

\theoremstyle{remark}

\newcommand{\abs}[1]{\left\vert#1\right\vert}
\newcommand{\skp}[1]{\left\langle\,#1\,\right\rangle}
\newcommand{\spr}[1]{\left\langle\,#1\,\right\rangle}
\newcommand{\kl}[1]{\left(#1\right)}
\newcommand{\Kl}[1]{\left\{#1\right\}}

\definecolor{aog}{rgb}{0.0, 0.5, 0.0}

\newcommand{\R}{\mathbb{R}} 
\newcommand{\C}{\mathbb{C}}
\newcommand{\N}{\mathbb{N}}
\newcommand{\Z}{\mathbb{Z}}
\newcommand{\Q}{\mathbb{Q}}

\newcommand{\vphi}{\varphi}

\newcommand{\clg}{c_{\ell,g}}

\newcommand{\wjk}{w_{jk}}

\newcommand{\wjkl}{w_{jk,\ell}}
\newcommand{\wjklg}{w_{jk,\ell g}}

\newcommand{\D}{\mathcal{D}}

\newcommand{\sumg}{\sum_{g=1}^G}
\newcommand{\sumjk}{\sum_{j,k\in\Z}}

\newcommand{\NGS}{{\operatorname{NGS}}}
\newcommand{\LGS}{{\operatorname{LGS}}}

\usepackage{pgfplots}
\usepackage{pgfplotstable,booktabs}

\newenvironment{customlegend}[1][]{
	\begingroup
	\csname pgfplots@init@cleared@structures\endcsname
	\pgfplotsset{#1}
}{
	\csname pgfplots@createlegend\endcsname
	\endgroup
}
\def\addlegendimage{\csname pgfplots@addlegendimage\endcsname}
\pgfkeys{/pgfplots/number in legend/.style={
		/pgfplots/legend image code/.code={
			\node at (0.295,-0.0225){#1};
		},
	},
}

\title{Singular Value and Frame Decomposition-based Reconstruction for Atmospheric Tomography}

\author{Lukas Weissinger\thanks{Johann Radon Institute Linz, Altenbergerstra{\ss}e 69, A-4040 Linz, Austria,
  (lukas.weissinger@ricam.oeaw.ac.at, simon.hubmer@ricam.oeaw.ac.at, bernadett.stadler@indmath.uni-linz.ac.at, ronny.ramlau@ricam.oeaw.ac.at).}
\and Simon Hubmer\footnotemark[1]
\and Bernadett Stadler\footnotemark[1]
\and Ronny Ramlau\footnotemark[1]~\thanks{Johannes Kepler University Linz, Institute of Industrial Mathematics, Altenbergerstra{\ss}e 69, A-4040 Linz, Austria, (ronny.ramlau@jku.at).}
}

\begin{document}
\maketitle
\begin{abstract}
Atmospheric tomography, the problem of reconstructing atmospheric turbulence profiles from wavefront sensor measurements, is an integral part of many adaptive optics systems used for enhancing the image quality of ground-based telescopes. Singular-value and frame decompositions of the underlying atmospheric tomography operator can reveal useful analytical information on this inverse problem, as well as serve as the basis of efficient numerical reconstruction algorithms. In this paper, we extend existing singular value decompositions to more realistic Sobolev settings including weighted inner products, and derive an explicit representation of a frame-based (approximate) solution operator. These investigations form the basis of efficient numerical solution methods, which we analyze via numerical simulations for the challenging, real-world Adaptive Optics system of the Extremely Large Telescope using the entirely MATLAB-based simulation tool MOST.
\end{abstract}


%===========================================================================
%===========================================================================
% % % % % % % % % % % % % % % % % % % % % % %  
% % % % % % Section - Introduction  % % % % %
% % % % % % % % % % % % % % % % % % % % % % %  
\section{Introduction}\label{sec:introduction}

\begin{figure}[ht!]
	\centering
	\includegraphics[width=0.45\textwidth]{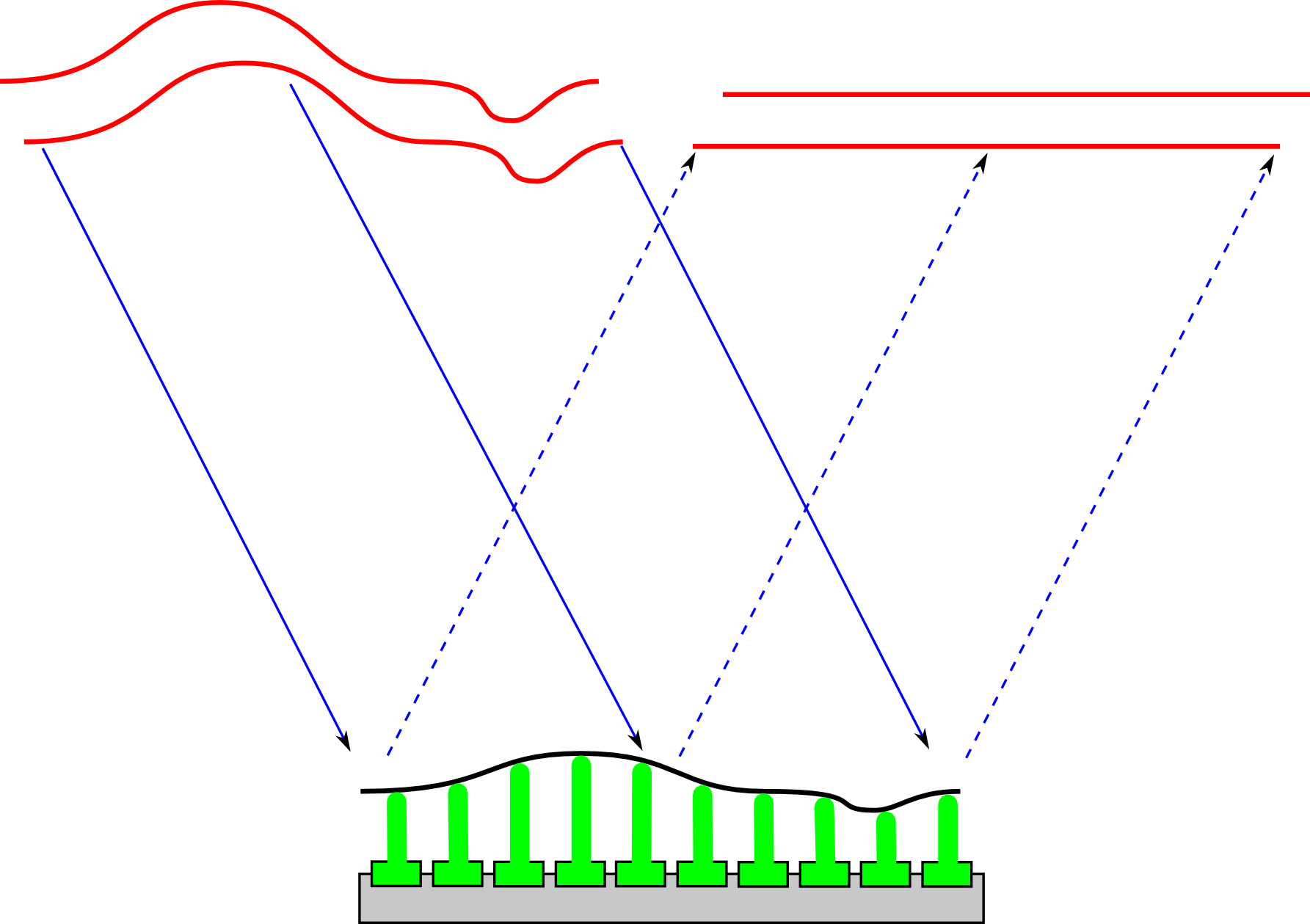}
	\qquad\qquad
	\includegraphics[width=0.40\textwidth, trim={6cm 0cm 6cm 0cm}, clip ]{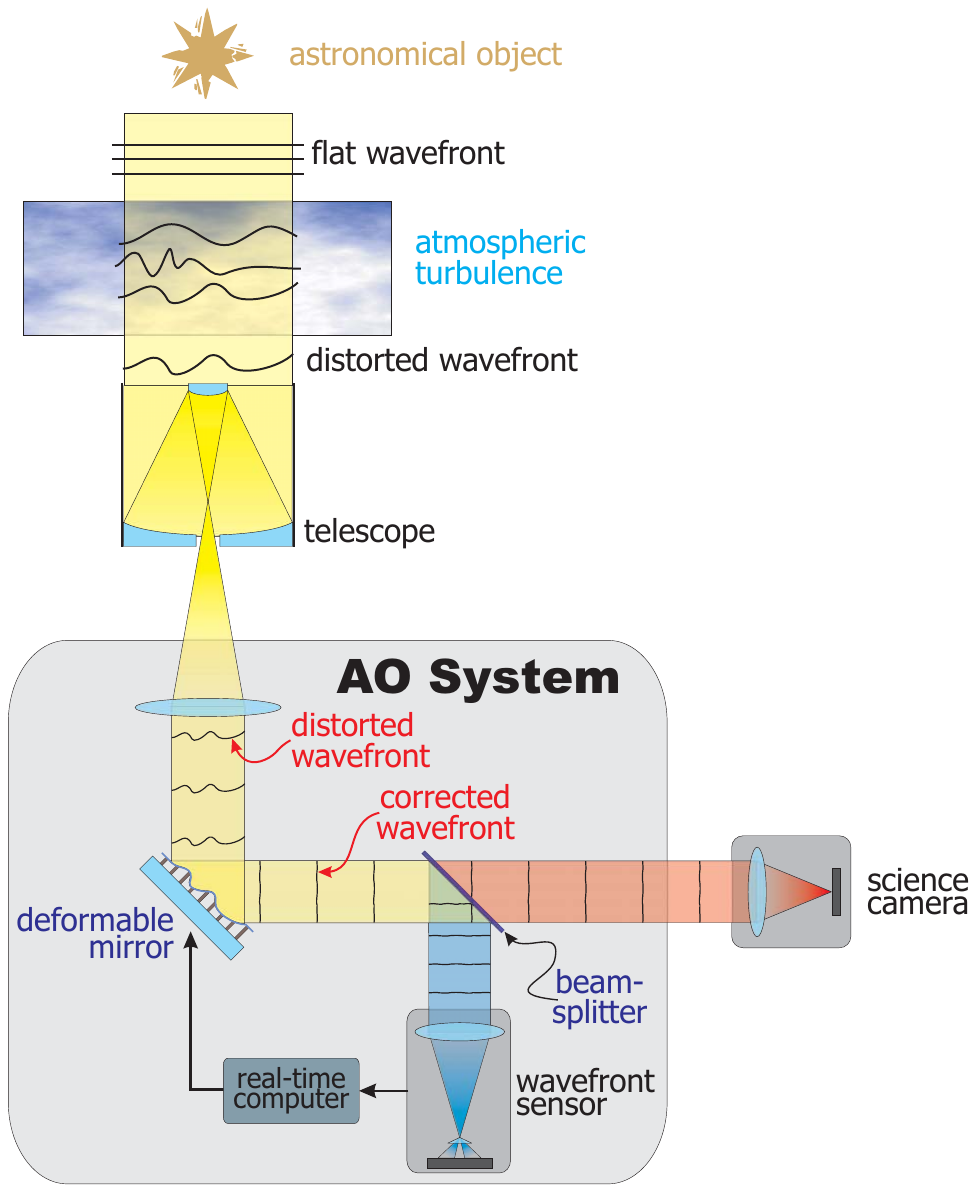}
	\caption{Schematic depiction of wavefront correction via deformable mirror (left, image from \cite{Auzinger_2017}) and working principle of a SCAO system (right, image from \cite{Egner_2006}).}
	\label{fig_AO_system}
\end{figure}

Adaptive Optics (AO) is an essential component of ground-based telescopes such as the Extremely Large Telescope (ELT) \cite{ELT_2020} of the European Southern Observatory (ESO), currently under construction in the Atacama desert in Chile. This is because temperature fluctuations in the atmosphere cause turbulence, which results in wavefront aberrations of the incoming light before it reaches the telescope. If not corrected by an AO system, these aberrations result in blurred images and therefore a severe loss of image quality. The general working principle of an AO system can be summarized as follows \cite{Roddier1999,Roggemann_Welsh_1996,Ellerbroek_Vogel_2009}: First, a wavefront sensor (WFS) is used to measure the wavefront aberration of the incoming light of some reference light source such as a natural guide stars (NGS). Then, a deformable mirror (DM) located in the light path is adjusted, such that after reflection from this DM, the incoming wavefront is approximately plane and thus aberration free; see Figure~\ref{fig_AO_system} (left). Since the atmosphere is rapidly changing, this measurement and correction cycle has to be repeated continuously at a frequency of about 500~Hz.

\begin{figure}[ht!]
	\centering
	\includegraphics[width=1\textwidth, clip, trim={0cm 0cm 0cm 0cm}]{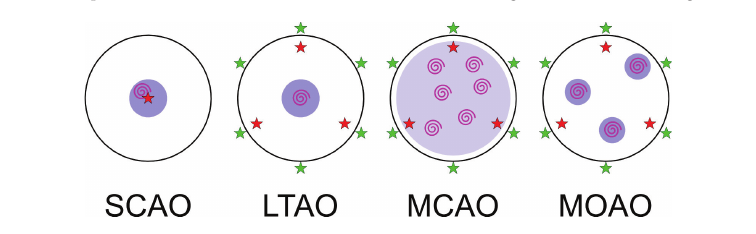}
	\caption{Schematic depiction of different types of AO systems. Magenta spirals represent astronomical objects of interest, while red and greens stars correspond to locations of NGS and LGS, respectively. The darker shaded areas correspond to the directions corrected for by the corresponding AO systems. Image taken from \cite{Auzinger_2017}.}
	\label{fig_AO_systems}
\end{figure}

The AO correction procedure outlined above forms the basis of so-called Single Conjugate Adaptive Optics (SCAO), which is used to observe astronomical objects in the near vicinity of a bright reference star; see Figure~\ref{fig_AO_system} (left) for a schematic depiction. However, if the distance between the object of interest and this NGS is too large, the image quality strongly deteriorates. This is due to the directional dependence of the wavefront aberration, caused by different atmospheric turbulence. Therefore, the WFS measurement in guide star direction is then no longer close to that of the observed object. A common remedy for this problem is to consider multiple guide stars, both NGS and artificial laser guide stars (LGS) created by powerful lasers in the sodium layer of the atmosphere. The incoming wavefront of each of those guide stars is measured by a separate WFS, from which one then aims to reconstruct the entire turbulence volume above the telescope. This is the atmospheric tomography problem considered in this paper. Once the atmospheric turbulence has been reconstructed, it is then possible to use one or several DMs to correct for objects with no guide star nearby, or obtain a correction over either a larger field of view (FoV) or in several view directions. These settings correspond to Laser Tomography Adaptive Optics (LTAO), Multiconjugate Adaptive Optics (MCAO), and Multiobject Adaptive Optics (MOAO), respectively, which are illustrated in Figure~\ref{fig_AO_systems}. For further details see e.g.\ \cite{Diolaiti_et_al_2016,Andersen_Eikenberry_Fletcher_Gardhuose_Leckie_Veran_Gavel_Clare_Jolissaint_Julian_Rambold_2006,Hammer_Sayede_Gendron_Fusco_Burgarella_Cayatte_Conan_Courbin_Flores_Guinouard_2002,Puech_Flores_Lehnert_Neichel_Fusco_Rosati_Cuby_Rousset_2008,Rigaut_Ellerbroek_Flicker_2000}.

\begin{figure}[ht!]
	\centering
	\includegraphics[width=0.45\textwidth]{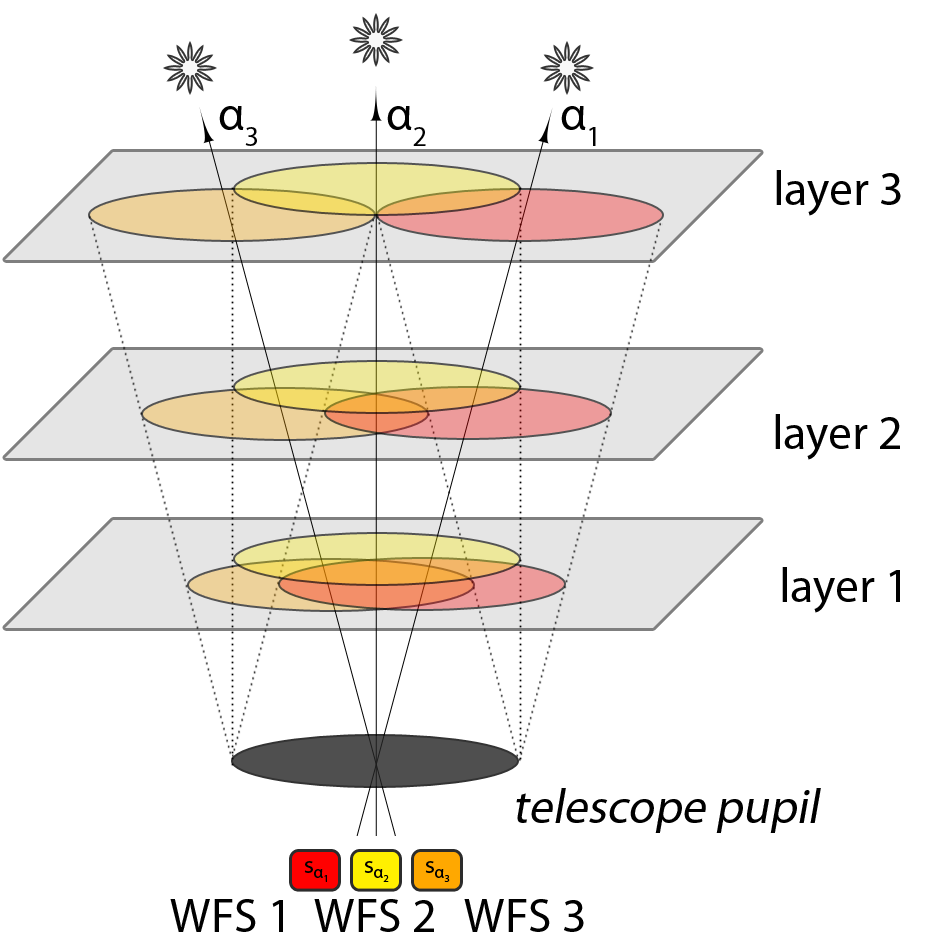}
	\qquad
	\includegraphics[width=0.45\textwidth]{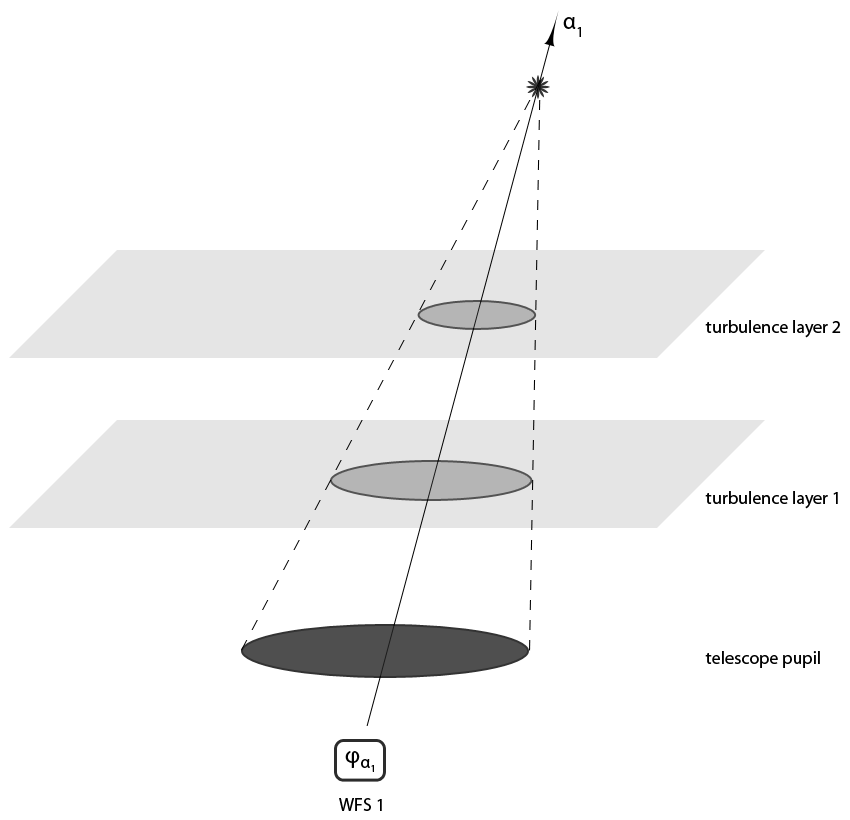}
	\caption{Illustration of the atmospheric tomography problem with three turbulence layers, NGSs and corresponding WFSs (left). Light stemming from a single LGS is influenced by the cone effect (right). Images taken from \cite{Yudytskiy_2014}.}
	\label{fig_AO_Tomo}
\end{figure}

In the form introduced above, atmospheric tomography falls into the category of limited-angle tomography problems: only a small number of guide stars is available ($6$ LGS for the ELT), and the NGS and LGS only have a small angle of separation ($1$ arcmin for MCAO and $3.5$ arcmin for MOAO). As such, it is severely ill-posed and thus practically infeasible without additional restrictions \cite{Davison_1983, Natterer_2001}. These come in the form of assumptions on the structure of the atmosphere, more precisely that it consists of a finite number of (infinitely) thin turbulent layers located at predefined heights. Atmospheric tomography then reduces to reconstructing the atmospheric turbulence profile on those finitely many layers from the given WFS measurements. An example of this problem for the case of three layers and three (natural) guide stars with corresponding WFSs is depicted in Figure~\ref{fig_AO_Tomo} (left).

The atmospheric tomography problem has attracted considerable attention in the past. Among the many proposed reconstruction methods we mention e.g.\ the minimum mean square error method \cite{Fusco_Conan_Rousset_Mugnier_Michau_2001}, the back-projection
algorithm of \cite{Gavel_2004}, conjugate gradient type
iterative reconstruction methods with suitable preconditioning \cite{Ellerbroek_Gilles_Vogel_2003,Gilles_Ellerbroek_Vogel_2003,Gilles_Ellerbroeck_2008,Yang_Vogel_Ellerbroek_2006,Vogel_Yang_2006}, the Fractal Iterative Method (FrIM) \cite{Tallon_TallonBosc_Bechet_Momey_Fradin_Thiebaut_2010,Tallon_Bechet_TallonBosc_Louarn_Thiebaut_Clare_Marchetti_2012,Thiebaut_Tallon_2010}, the Finite Element Wavelet Hybrid Algorithm (FEWHA) \cite{Yudytskiy_2014,Yudytskiy_Helin_Ramlau_2013,Yudytskiy_Helin_Ramlau_2014,Stadler2020,Stadler2021}, as well as a Kaczmarz iteration \cite{Ramlau_Rosensteiner_2012, Rosensteiner_Ramlau_2013}. For further methods as well as important practical considerations see also \cite{Ellerbroek_Gilles_Vogel_2002,Gilles_Ellerbroek_Vogel_2002,Gilles_Ellerbroek_Vogel_2007,Poettinger_Ramlau_Auzinger_2019,Raffetseder_Ramlau_Yudytskiy_2016,Ramlau_Obereder_Rosensteiner_Saxenhuber_2014,Saxenhuber_Ramlau_2016} and the references therein. In recent years FEWHA has already proven in simulations to provide an excellent reconstruction quality in real-time for the MORFEO instrument of the ELT \cite{Stadler2022}. A clever discretization strategy together with a matrix-free implementation and a small number of conjugate gradient iterations makes reconstructions in real-time possible and enables on the fly parameter updates.

While these methods have all been studied in detail both analytically and numerically, they do not yield much new knowledge about the atmospheric tomography problem itself. This should be contrasted with the (limited-angle) tomography operator, from which the atmospheric tomography operator is derived \cite{Ellerbroek_Vogel_2009,Fusco_Conan_Rousset_Mugnier_Michau_2001}. There, a large number of theoretical results offer insight into the structure and ill-posedness of the classic tomography problem \cite{Natterer_2001}. Many of these are directly related to the availability of a singular value decomposition (SVD) of the (limited-angle) Radon transform \cite{Davison_1983,Natterer_2001}. Motivated by this, a singular value-type decomposition (SVTD) of the atmospheric tomography operator has recently been derived in \cite{Neubauer_Ramlau_2017}, and has provided the first theoretical insights into the ill-posedness of the atmospheric tomography problem. However, this SVTD is only valid on a square telescope aperture in a NGS-only setting, which limits its practical applicability. This motivated the derivation of a frame decomposition (FD) in \cite{Hubmer_Ramlau_2020}, which provided an SVD-like decomposition of the atmospheric tomography operator valid for general aperture shapes and a mixture of both NGS and LGS. The main drawback of this approach is that it only provides an approximate solution to the problem; see also \cite{Weissinger2021}. Note that in \cite{Hubmer_Ramlau_2020} an SVTD was also derived for square apertures and a LGS-only setting. Further analytic properties including the non-uniqueness of the atmospheric tomography problem have recently been considered in \cite{Ramlau_Stadler_2024}.

In this paper, we aim to advance the study of the atmospheric tomography problem in two ways: First, we extend the previously derived SVTDs in the NGS-only and LGS-only cases to a more physically realistic setting. In particular, we consider real-order Sobolev spaces for the definition space of the operator, which is motivated by the Kolmogorov turbulence model for the atmospheric layers \cite{Kolmogorov}. Furthermore, we incorporate commonly used turbulence profiles into the SVTDs by replacing the standard inner products with correspondingly weighted version. Furthermore, we consider a split-tomography approach to solve the atmospheric tomography problem in the mixed NGS/LGS problem using the derived SVTDs. Secondly, for the FD of the atmospheric tomography operator, we derive an explicit representation of the involved dual frame functions, which previously had to be computed numerically. This in turn leads to an explicit representation of the (approximate) frame inverse, which allows for a highly efficient implementation. The final contribution of this paper is a numerical comparison of the SVTDs and FD with state-of-the-art reconstruction algorithms in a realistic adaptive optics simulation environment.

The outline of this paper is as follows: In Section~\ref{sec:tomo}, we recall the definition and some basic properties of the atmospheric tomography operator. In Section~\ref{sec:svtd}, we then derive a singular value decomposition of this operator in a realistic Sobolev space setting including general weighted inner products incorporating turbulence profiles. In Section~\ref{sec:fd}, we then consider a frame decomposition of the atmospheric tomography operator, and derive an explicit representation of the involved dual frame functions and the corresponding (approximate) solution operator. Finally, in Section~\ref{sec:fd}, we numerically test the resulting reconstruction methods in a realistic environment using the adaptive optics simulation tool MOST, and compare the results to those obtained with two other state-of-the-art reconstruction algorithms.

% % % % % % % % % % % % % % % % % % % % % % % % %
% Section - The Atmospheric Tomography Operator %
% % % % % % % % % % % % % % % % % % % % % % % % %
\section{The Atmospheric Tomography Operator}\label{sec:tomo}

In this section, we recall the definition and some basic properties of the atmospheric tomography operator, which has originally been derived from the (limited-angle) Radon transform using the layered structure of the atmospheric turbulence \cite{Ellerbroek_Vogel_2009,Fusco_Conan_Rousset_Mugnier_Michau_2001}.

First, let the domain $\Omega_A \subset \R^2$ represent the telescope aperture, which typically (but not always) is a circular and symmetric domain centered around the origin. Furthermore, assume that there are $L$ atmospheric layers, i.e., planes parallel to the aperture $\Omega_A$, located at distinct heights $h_{\ell} \in \R_0^+$ for $\ell = 1 \,, \dots \,, L$. We assume that the heights are given in ascending order and note that typically $h_1 = 0$. Note that all 2-dimensional domains defined here are embedded in $\R^3$ by fixing the z-coordinate. Next, consider $G$ different guide stars with corresponding direction vectors $\alpha_g = (\alpha_g^x,\alpha_g^y) \in \R^2$ for $g = 1\ ,,\dots \,, G$. The vectors $\alpha_g$ are such that seen from the center of the telescope aperture, the vectors $(\alpha_g^x,\alpha_g^y,1) \in \R^3$ point directly at the corresponding guide stars. Now, assume that the first $G_\NGS$ guide stars are NGS, while the remaining $G_\LGS$ guide stars are LGS, such that $G= G_\NGS + G_\LGS$. Then we can define the coefficients
    \begin{equation*}
        c_{\ell,g}=
        \begin{cases}
        1\,,\quad & g\in\{1\,,\dots\,,G_{\NGS}\} \,,
        \\ 1- h_\ell/h_{\LGS} \,, \quad & g\in\{G_{\NGS}+1\,,\dots\,,G\} \,,
        \end{cases}
    \end{equation*}
where $h_\LGS$ denotes the height of the sodium layer in the atmosphere in which LGS are created (approximately $90$ km). The coefficients $c_{\ell,g}$ model the cone effect for LGS as illustrated in Figure~\ref{fig_AO_Tomo} (right). Note that since $h_\ell < h_{LGS}$ for all $\ell=1,\ldots,L$, we have $c_{\ell,g} \in (0,1]$. Next, we introduce the domains 
    \begin{equation*}\label{Omegaelldef}
        \Omega_{\ell}:=\bigcup_{g=1}^G \Omega_A(\alpha_gh_\ell) \subset \R^2 \,,
        \qquad \forall \, \ell = 1\,,\dots \,, L \,,
    \end{equation*}
where
    \begin{equation*}
        \Omega_A(\alpha_gh_\ell):=\left\{r\in\mathbb{R}^2 \, \Big\vert \, \frac{r-\alpha_gh_\ell}{c_{\ell,g}}\in\Omega_A\right\} \,.
    \end{equation*}
The domains $\Omega_{\ell}$ are exactly those parts of the atmospheric layers which are ``seen'' by the WFSs. Consequently, these are also the only parts of the atmosphere which one can hope to reconstruct accurately. In the example shown in Figure~\ref{fig_AO_Tomo} (left), the domains $\Omega_l$ correspond to the union of the coloured areas; see also Figure~\ref{fig:domains}.

%%% Figure %%%
\def\centercircle{(0,0) circle (2cm)}
\def\firstcircle{(30:1cm) circle (1.6cm)}
\def\secondcircle{(90:1cm) circle (1.6cm)}
\def\thirdcircle{(150:1cm) circle (1.6cm)}
\def\fourthcircle{(210:1cm) circle (1.6cm)}
\def\fifthcircle{(270:1cm) circle (1.6cm)}
\def\sixthcircle{(330:1cm) circle (1.6cm)}
\def\quadra{(-2.8,2.8) rectangle (2.8,-2.8)}
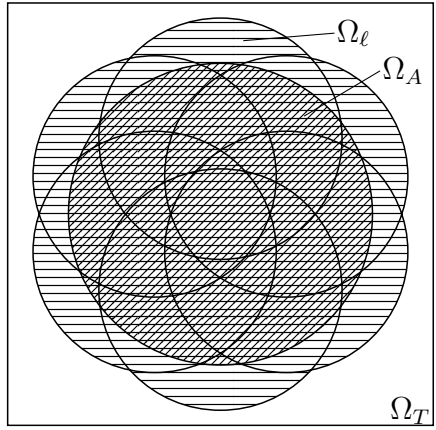
\begin{figure}[ht!]
    \centering 
\begin{tikzpicture}
 \begin{scope}
            \clip \centercircle \quadra;
        \fill[pattern=horizontal lines] \firstcircle \secondcircle \thirdcircle \fourthcircle \fifthcircle \sixthcircle;
        \fill[pattern=north east lines] \centercircle;
\end{scope}
\draw[line width=0.2mm]\centercircle\firstcircle\secondcircle\thirdcircle\fourthcircle\fifthcircle\sixthcircle\quadra;
\node at (2.5,-2.6) {$\Omega_T$};
\draw (1.1,1.3) -- (2.1,1.9);
\node at (2.4,1.9) {$\Omega_A$};
\draw (0.3,2.3) -- (1.5,2.4);
\node at (1.75,2.4) {$\Omega_\ell$};
\label{fig:}
\end{tikzpicture}
    \caption{A schematic drawing of the domains $\Omega_A$, $\Omega_\ell,$ and $\Omega_T$.}
    \label{fig:domains}
\end{figure}
%%% End Figure %%%

We now want to define the atmospheric tomography operator between suitable Lebesgue spaces. For this, we introduce the space $L_2(\Omega,\gamma)$ with $\gamma > 0$, by which we simply mean the classic Lebesgue space $L_2(\Omega)$ equipped with the scaled inner product 
    \begin{equation*}
        \spr{ u,v }_{L_2(\Omega,\gamma)} :=\frac{1}{\gamma}\int_{\Omega} u(r)\overline{v(r)}\mathrm{d}r \,.
    \end{equation*}
Let $\phi=(\phi_\ell)_{\ell=1,\ldots,L}$ denote the refractive index variations (a dimensionless quantity), related to temperature fluctuations within the atmosphere, causing atmospheric turbulence, and let $\varphi=(\varphi_g)_{g=1,\ldots,G}$ denote the incoming wavefronts as reconstructed by the WFSs. Then for given turbulence weights $(\gamma_\ell)_{\ell=1,\dots,L}$ the atmospheric tomography operator can be defined as \cite{Fusco_Conan_Rousset_Mugnier_Michau_2001}
    \begin{equation}\label{defA}
    \begin{aligned}
         &A:\mathcal{D}(A):=\prod_{\ell=1}^L L_2(\Omega_{\ell},\gamma_\ell)\to L_2(\Omega_A)^G, \qquad \phi = (\phi_{\ell})_{\ell=1}^L \mapsto \varphi =\kl{(A_g \phi) }_{g=1}^G\,,
        \\  
        &\qquad(A_g \phi)(r) := \sum_{\ell=1}^L\phi_l(c_{\ell,g} r + \, \alpha_g h_{\ell}) \,.
    \end{aligned}
    \end{equation}
Note that the product spaces above are equipped with the canonical inner products
    \begin{equation*}\label{innerproducts}
        \spr{ \phi,\psi }_{\prod_{\ell=1}^L L_2(\Omega_\ell,\gamma_\ell)}
        =\sum_{\ell=1}^L \spr{ \phi_\ell,\psi_\ell }_{L_2(\Omega_\ell,\gamma_\ell)} 
        =\sum_{\ell=1}^L \frac{1}{\gamma_\ell} \spr{ \phi_\ell,\psi_\ell }_{L_2(\Omega_\ell)}\,,
    \end{equation*}
and
    \begin{equation*}
        \spr{ \varphi,\theta }_{L_2(\Omega_A)^G}=\sum_{g=1}^G \spr{ \varphi_g,\theta_g }_{L_2(\Omega_A)} \,.
    \end{equation*}
The atmospheric tomography operator $A$ as defined in \eqref{defA} essentially only sums up the contributions of each turbulence layer in the direction of the guide stars. The weights $\gamma_\ell$ are used to place a higher emphasis on layers on which a strong turbulence is expected, and are assumed to satisfy $\sum_{\ell=1}^L \gamma_\ell=1$. In practice, they are known quantities derived from previously measured turbulence profiles. Note that from a mathematical point of view, one may also consider non-constant (i.e.\ spatially varying) weights $\gamma_\ell$, thereby placing a varying emphasis on certain areas within the atmospheric layers. An example are the piecewise constant weights introduced in \cite{Saxenhuber_Ramlau_2016}, which are also briefly discussed at the end of Section~\ref{sec:fd}. Since they are not commonly used in practice, we do not consider such spatially varying weights here, but note that our subsequent analysis could be extended to include them as well.

For further use, we also require the adjoint of the atmospheric tomography operator $A$, which following \cite[Proposition 2]{Ramlau_Rosensteiner_2012} can be derived to be
    \begin{equation}\label{Adjoint}
    \begin{split}
        (A^\ast \varphi)(r) & = 
        \sum_{g=1}^G (A_g^* \vphi)(r) 
        =
        \sum_{g=1}^G\kl{\frac{\gamma_\ell}{c_{\ell,g}^2}\varphi_g\left(\frac{r-\alpha_g h_\ell}{c_{\ell,g}}\right)I_{\Omega_A(\alpha_gh_\ell)}(r) }_{l=1}^L \,,
    \end{split}
    \end{equation}
where $I_{\Omega_A(\alpha_gh_\ell)}$ denotes the indicator function of the domain $\Omega_A(\alpha_gh_\ell)$.
It was shown in \cite{Neubauer_Ramlau_2017} that the atmospheric tomography operator $A$ is not compact, and thus an SVD does not necessarily need to exist. Nevertheless, an SVTD can be derived for a ``periodic'' atmospheric tomography operator on a square domain, which we now introduce in a general Sobolev setting. For this, let $\Omega_T=[-T,T]^2 \subset \R^2$ be a square domain with $T$ chosen sufficiently large such that, cf.~Figure~\ref{fig:domains}, 
    \begin{equation}\label{Tbed}
        \Omega_A+\alpha_g h_\ell \subset c_{\ell} \Omega_T\,, \qquad \forall \, g=1, \ldots, G\,, \quad \forall \, \ell=1, \ldots, L \,,
    \end{equation}
where 
    \begin{equation*}
        c_{\ell} := \min_{g=1,\ldots,G}\Kl{c_{\ell,g}}\,, \qquad \forall \, \ell=1, \ldots, L \,.
    \end{equation*}
Then, we introduce the \emph{periodic} Sobolev spaces $H^s\left(c_{\ell} \Omega_T,\gamma_\ell\right)$ via the inner product 
    \begin{equation}\label{sob_norm}
        \spr{u,v}_{H^s(c_{\ell}\Omega_T,\gamma_\ell)}:=\sum_{j,k\in\Z}\left(1+\beta_{\ell,T}|(j,k)|^2\right)^{s}u_{jk,\ell }\overline{v_{jk,\ell }} \,,
        \qquad 
        \beta_{\ell,T} = \pi^2(c_{\ell} T)^{-2}\,,
    \end{equation}
where $u_{jk,\ell}:=\skp{u,w_{jk,\ell}}_{L_2(c_{\ell} \Omega_T,\gamma_\ell)}$ and $v_{jk,\ell}:=\skp{v,w_{jk,\ell }}_{L_2(c_{\ell} \Omega_T,\gamma_\ell)}$ with
    \begin{equation}\label{def_wjk_wjkl}
        w_{jk,\ell}(x,y) :=\frac{\gamma_\ell^{1/2}}{c_{\ell}}w_{jk}((x,y)/c_{\ell})\,,
        \quad \text{and} \quad 
        w_{jk}(x,y):=\frac{1}{2T}e^{i\omega(jx+ky)} \,, \quad \omega=\frac{\pi}{T} \,.
    \end{equation}
The functions $w_{jk}$ and $w_{jk,\ell}$ defined in \eqref{def_wjk_wjkl} form orthonormal bases over the spaces $L_2(\Omega_T)$ and $L_2(c_{\ell} \Omega_T,\gamma_{\ell})$, respectively. Note that our definition of $H^s\left(c_{\ell} \Omega_T,\gamma_\ell\right)$ amounts to the classic Fourier-series definition of periodic Sobolev spaces (see e.g.~\cite{Hubmer_Sherina_Ramlau_2023,Ramlau_Teschke_2004_1}), adapted to our specific scaled domains $c_{\ell} \Omega_T$ and incorporating the turbulence weights $\gamma_{\ell}$. For integer order $s \in \N$, the inner product \eqref{sob_norm} is equivalent to the classic Sobolev space inner product \cite[Proposition~7.2]{Hubmer_Sherina_Ramlau_2023}, while for general $s\in\R$ equivalence holds over the Triebel spaces, which include zero boundary conditions \cite{Natterer_2001}. With this, the periodic atmospheric tomography operator is now defined as
    \begin{equation*}
    \begin{split}
        &\Tilde{A}^{(s)}: \prod_{\ell=1}^L H^s\left(c_\ell \Omega_T,\gamma_\ell\right) \rightarrow L_2\left(\Omega_T\right)^G \,, \qquad \phi = (\phi_{\ell})_{\ell=1}^L \mapsto \varphi =\kl{(\Tilde{A}^{(s)}_g \phi) }_{g=1}^G
        \\
        &\qquad(\Tilde{A}^{(s)}_g \phi) (r):= \sum_{\ell=1}^L\phi_{\ell}(c_{\ell,g} r + \, \alpha_g h_{\ell})  \,,
    \end{split}
    \end{equation*}
with $0\leq s \in \R$.
The compactness of this operator now depends on the specific choice of $s$. If $s=0$, then as above it can be shown that $\Tilde{A}^{(0)}$ is not compact, and thus an SVD does not necessarily need to exist. However, if $s > 0$, then the compactness of the Sobolev embedding operator on the bounded domain $\Omega_T$ implies that $\Tilde{A}^{(s)}$ is compact, and thus an SVD exists \cite{Engl_Hanke_Neubauer_1996}. In the specific case of $s=0$ and $\gamma_{\ell} = 1$ for all $\ell=1,\ldots,L$, SVTDs for the NGS-only and LGS-only case have been derived in \cite{Neubauer_Ramlau_2017} and \cite{Hubmer_Ramlau_2020}, respectively. As we shall see below, these leverage algebraic properties of the scaled exponential functions $w_{jk}$. In particular, it was shown in \cite[Theorem~5.1]{Neubauer_Ramlau_2017} that if
    \begin{equation}\label{cond_wellposed}
        \frac{\alpha_g^x h_{\ell}}{T} \in \Q \,, 
        \quad \text{and} \quad 
        \frac{\alpha_g^x h_{\ell}}{T} \in \Q \,,
        \qquad
        \forall \, g = 1,\ldots,G \,, \quad \forall \, \ell = 1, \ldots ,L \,,
    \end{equation}
then the pseudo-inverse $(\Tilde{A}^{(0)})^\dagger$ is bounded. Hence, in this setting the (periodic) atmospheric tomography problem is well-posed. Furthermore, one can find examples violating \eqref{cond_wellposed} which lead to unboundedness of $(\Tilde{A}^{(0)})^\dagger$. On the other hand, for $s > 0$ the compactness of $\Tilde{A}^{(s)}$ implies that $(\Tilde{A}^{(s)})^\dagger$ is always unbounded, and thus the (periodic) atmospheric tomography problem is always ill-posed in this case.

% % % % % % % % % % % % % % % % % % % % % % %  
% % % % % % Section - SVD % % % % % % % % % %
% % % % % % % % % % % % % % % % % % % % % % %  
\section{Singular Value Type Decompositions}\label{sec:svtd}

In this section, we derive an SVTD for the periodic atmospheric tomography operator $\Tilde{A}^{(s)}$ for the case of either NGS-only or LGS-only. In this case, $c_{\ell,g} = c_{\ell}$ and thus
    \begin{equation}\label{defAperonly}
    \begin{split}
        &\Tilde{A}^{(s)}: \D(\Tilde{A}^{(s)}) := \prod_{\ell=1}^L H^s\left(c_\ell \Omega_T,\gamma_\ell\right) \rightarrow L_2\left(\Omega_T\right)^G \,,\quad \phi = (\phi_{\ell})_{\ell=1}^L \mapsto \varphi =\kl{(\Tilde{A}^{(s)}_g \phi) }_{g=1}^G,
        \\
        & \qquad(\Tilde{A}^{(s)}_g \phi)(r)= \sum_{\ell=1}^L\phi_{\ell}(c_{\ell} r + \, \alpha_g h_{\ell})  \,.
    \end{split}
    \end{equation}
The upcoming analysis closely follows ideas from \cite{Neubauer_Ramlau_2017,Hubmer_Ramlau_2020}, and is based on the fact that the functions $w_{jk}$ and $w_{jk,\ell}$ defined in \eqref{def_wjk_wjkl} form orthonormal bases over the spaces $L_2(\Omega_T)$ and $L_2(c_{\ell} \Omega_T,\gamma_{\ell})$, respectively. This implies the following result.

\begin{lemma}
Let $s\geq 0$ and $\ell \in \{ 1, \ldots , L \}$ be arbitrary but fixed. Then the functions 
    \begin{equation}\label{vjkl}
        w^{(s)}_{jk,\ell}:=\left(1+\beta_{\ell,T}|(j,k)|^2\right)^{-s/2} w_{jk,\ell} \,,
    \end{equation}
with $w_{jk,\ell}$ as in \eqref{def_wjk_wjkl} form an orthonormal basis for the Sobolev-space $H^s(c_\ell\Omega_T,\gamma_\ell)$.
\end{lemma}
\begin{proof}
Let $s\geq 0$ and $\ell \in \{ 1, \ldots , L \}$ be arbitrary but fixed. Since the functions $w_{jk,\ell}$ form an orthonormal basis of $L_2(c_{\ell}\Omega_T,\gamma_{\ell})$, it follows with \eqref{sob_norm} that
    \begin{equation*}
    \begin{aligned}
        &\skp{w^{(s)}_{j'k',\ell},w^{(s)}_{j''k'',\ell}}_{H^s(c_\ell\Omega_T,\gamma_\ell)}
        \\
        & \qquad \overset{\eqref{sob_norm}}{=} 
        \sum_{j,k\in\Z}\left(1+\frac{\pi^2|(j,k)|^2}{(c_{\ell} T)^2}\right)^{s}\skp{w^{(s)}_{j'k',\ell},w_{jk,\ell}}_{L_2(c_{\ell} \Omega_T,\gamma_\ell)}\overline{\skp{w^{(s)}_{j''k'',\ell},w_{jk,\ell }}_{L_2(c_{\ell} \Omega_T,\gamma_\ell)}} 
        \\
        & \qquad \overset{\eqref{vjkl}}{=} 
        \sum_{j,k\in\Z} \skp{w_{j'k',\ell},w_{jk,\ell}}_{L_2(c_{\ell} \Omega_T,\gamma_\ell)}\overline{\skp{w_{j''k'',\ell},w_{jk,\ell }}_{L_2(c_{\ell} \Omega_T,\gamma_\ell)}}
        = 
        \delta_{j',j''} \delta_{k',k''} \,,
    \end{aligned}
    \end{equation*}
and thus the functions $w_{jk,\ell}^{(s)}$ are orthonormal in $H^s(c_\ell\Omega_T,\gamma_\ell)$. In order to show that they are also a basis, note that $H^s(c_\ell\Omega_T,\gamma_\ell)\subset L_2(c_\ell\Omega_T,\gamma_\ell)$, and thus for each $u \in H^s(c_\ell\Omega_T,\gamma_\ell)$ there holds
    \begin{equation*}
    \begin{aligned}
        u&=\sum_{j,k\in \Z}\skp{u,w_{jk,\ell}}_{L_2(c_\ell\Omega_T,\gamma_\ell)}w_{jk,\ell}
        \\
        &\overset{\eqref{sob_norm}}{=}\sum_{j,k\in \Z}\left(1+\beta_{\ell,T}|(j',k')|^2\right)^{-s/2}\skp{u,w^{(s)}_{jk,\ell}}_{H^s(c_\ell\Omega_T,\gamma_\ell)}w_{jk,\ell}
        \\
        &\overset{\eqref{vjkl}}{=}\sum_{j,k\in \Z}\skp{u,w^{(s)}_{jk,\ell}}_{H^s(c_    \ell\Omega_T,\gamma_\ell)}w^{(s)}_{jk,\ell} \,,
    \end{aligned}
    \end{equation*}
which yields completeness of $\{w^{(s)}_{jk,\ell} \}_{j,k\in\Z}$ in $H^s(c_\ell\Omega_T,\gamma_\ell)$ and concludes the proof. 
\end{proof}

Due to the above result, every $\phi_\ell\in H^s(c_\ell\Omega_T,\gamma_\ell)$ can be written in the form
    \begin{equation}\label{phiexpansion}
        \phi_\ell=\sum_{j,k\in\mathbb{Z}}\phi_{jk,\ell} \, w^{(s)}_{jk,\ell} \,,
        \qquad \text{where} \qquad 
        \phi_{jk,\ell} := \spr{\phi_\ell,w^{(s)}_{jk,\ell} }_{H^s(c_\ell\Omega_T,\gamma_\ell)} \,,
    \end{equation}
and thus for an arbitrary turbulence $\phi=(\phi_\ell)_{\ell=1}^L\in \D(\Tilde{A}^{(s)}) $ there holds
    \begin{equation*}
        \phi = \kl{ \sum_{j,k\in\mathbb{Z}}\phi_{jk,\ell} \, w^{(s)}_{jk,\ell} }_{\ell = 1}^L \,.
    \end{equation*}
Collecting the coefficients $\phi_{jk,\ell}$ into vectors $\phi_{jk}:=(\phi_{jk,1},\ldots,\phi_{jk,L}) \in\mathbb{C}^L$ we obtain

\begin{proposition}\label{matrixdecompA}
Let the periodic atmospheric tomography operator $\tilde{A}^{(s)}$ be defined as in \eqref{defAperonly}, let $w_{jk,\ell}^{(s)}$ be as in \eqref{vjkl}, and let the matrices $\tilde{A}_{jk}\in\mathbb{C}^{G\times L}$ be defined as
    \begin{equation}\label{Ajkdef}
        \tilde{A}^{(s)}_{jk} := \kl{ (2T)w^{(s)}_{jk,\ell}(\alpha_g^x h_\ell,\alpha_g^y h_\ell) }_{g,\ell=1}^{G,L} \,.
    \end{equation}
Then for all $\phi\in\D(\Tilde{A}^{(s)}) $ there holds
    \begin{equation}\label{eq_As_dec}
        (\tilde{A}^{(s)}\phi)(x,y)=\sum_{j,k\in\mathbb{Z}}\big(\tilde{A}^{(s)}_{jk}\phi_{jk}\big)w_{jk}(x,y) \,.
    \end{equation}
\end{proposition}
\begin{proof}
The proof of this proposition follows the lines of \cite[Proposition 4.1]{Hubmer_Ramlau_2020}: From the definition of $\tilde{A}^{(s)}$ and with the coefficient expansion \eqref{phiexpansion} it follows that
    \begin{equation*}
    \begin{aligned}
        (\tilde{A}_g^{(s)}\phi)(x,y)&=\sum_{\ell=1}^L \phi_\ell(c_\ell x+\alpha_g^x h_\ell,c_\ell y +\alpha_g^y h_\ell)
        \\
        &\overset{\eqref{phiexpansion}}{=} \sum_{\ell=1}^L\sum_{j,k\in\mathbb{Z}}\phi_{jk,\ell}w^{(s)}_{jk,\ell}(c_\ell x+\alpha_g^x h_\ell,c_\ell y +\alpha_g^y h_\ell)
        \\
        \overset{\eqref{def_wjk_wjkl},\eqref{vjkl}}&{=}\sum_{\ell=1}^L\sum_{j,k\in\mathbb{Z}}\phi_{jk,\ell}(2T)w_{jk}(x,y)w^{(s)}_{jk,\ell}(\alpha_g^x h_\ell,\alpha_g^y h_\ell)
        \\
        \overset{\eqref{Ajkdef}}&{=}\sum_{j,k\in\mathbb{Z}}\big(\tilde{A}^{(s)}_{jk}\phi_{jk}\big)_g w_{jk}(x,y) \,,
    \end{aligned}
    \end{equation*}
which yields \eqref{eq_As_dec}. Note that the interchanging of series in the last line is justified since the norm of all the matrices $\tilde{A}^{(s)}_{jk}$ is bounded independently of $j$ and $k$.
\end{proof}

Following \cite{Hubmer_Ramlau_2020, Neubauer_Ramlau_2017}, we now consider SVDs of the matrices $\tilde{A}^{(s)}_{jk}$. For all $j,k \in \Z$, let $r^{(s)}_{jk}$ denote the rank of $\tilde{A}^{(s)}_{jk}$, and let $u^{(s)}_{jk,n}\in\mathbb{C}^{G}$, $v^{(s)}_{jk,n}\in\mathbb{C}^L$, and  $\sigma^{(s)}_{jk,n}\in\mathbb{R^+}$,  $n=1,\ldots,r^{(s)}_{jk}\leq \min\{G,L\}$ be the singular vectors and values of $\tilde{A}^{(s)}_{jk}$ satisfying
    \begin{equation}\label{singularvd}
    \begin{aligned}
        \tilde{A}^{(s)}_{jk}\phi_{jk}&=\sum_{n=1}^{r^{(s)}_{jk}}\sigma^{(s)}_{jk,n}\left((v^{(s)}_{jk,n})^H\phi_{jk}\right)u^{(s)}_{jk,n} \,,
        \\
        (v^{(s)}_{jk,m})^H v^{(s)}_{jk,n}&=\delta_{mn} \,, \quad (u^{(s)}_{jk,m})^H u^{(s)}_{jk,n}=\delta_{mn} \,,
        \\ \vspace{0.4pt}
        \sigma^{(s)}_{jk,1}&\geq\sigma^{(s)}_{jk,2}\geq\ldots\geq\sigma^{(s)}_{jk,n}>0 \,,
    \end{aligned}
    \end{equation}
where the the superscript $H$ denotes the Hermitian of a matrix. Combining Proposition~\ref{matrixdecompA} with these SVDs \eqref{singularvd} we obtain the following SVTD of $\tilde{A}^{(s)}$:
    \begin{equation}\label{eq:Aopsingtype}
        (\tilde{A}^{(s)}\phi)(x,y)=\sum_{j,k\in\mathbb{Z}}\left(\sum_{n=1}^{r^{(s)}_{jk}}\sigma^{(s)}_{jk,n}\kl{(v^{(s)}_{jk,n})^H\phi_{jk}}u^{(s)}_{jk,n}\right)w_{jk}(x,y)\,.
    \end{equation}
As a result, we obtain an expression for the Moore-Penrose inverse $(\tilde{A}^{(s)})^\dagger$ of $\tilde{A}^{(s)}$.

\begin{theorem}
For $s \geq 0$ let the periodic atmospheric tomography operator $\tilde{A}^{(s)}$ be defined as in \eqref{defAperonly} and let \eqref{eq:Aopsingtype} be its SVTD. Furthermore, let $\varphi \in L_2(\Omega_T)^G$ with 
    \begin{equation*}
        \varphi = \sum_{j,k\in\mathbb{Z}}\varphi_{jk}w_{jk}\,, 
        \qquad \text{where} \qquad 
        \varphi_{jk} = \kl{\spr{\varphi_g,w_{jk}}_{L_2(\Omega_T)}}_{g=1}^G \in\mathbb{C}^G \,.
    \end{equation*}
Then the best approximate solution of the equation $\tilde{A}^{(s)}\phi=\varphi$ is given by 
    \begin{equation}\label{Aperdagger}
        ((\tilde{A}^{(s)})^\dagger \varphi)_\ell(x,y):=\sum_{j,k\in\mathbb{Z}}\left(\sum_{n=1}^{r^{(s)}_{jk}}\frac{\kl{(u^{(s)}_{jk,n})^H\varphi_{jk}}}{\sigma^{(s)}_{jk,n}}v^{(s)}_{jk,n}\right)_\ell w^{(s)}_{jk,\ell}(x,y) \,,
    \end{equation}
which is well-defined if and only if the following \textit{Picard condition} holds:
    \begin{equation*}
        \sum_{j,k\in\mathbb{Z}}\sum_{n=1}^{r^{(s)}_{jk}}\frac{|(u^{(s)}_{jk,n})^H\varphi_{jk}|^2}{(\sigma^{(s)}_{jk,n})^2}<\infty \,.   
    \end{equation*}
\end{theorem}
\begin{proof}
Let $s \geq 0$ be arbitrary but fixed and recall from its definition \eqref{phiexpansion} that
    \begin{equation*}
        \phi_{jk} = (\phi_{jk,1}, \ldots , \phi_{jk,L}) = \kl{ \spr{ \phi_\ell,w^{(s)}_{jk,\ell} }_{H^s(c_\ell\Omega_T,\gamma_\ell)}}_{\ell=1}^L \,.
    \end{equation*}
Inserting this into the decomposition \eqref{eq:Aopsingtype} we find that
    \begin{equation*}
        (\tilde{A}^{(s)}\phi)(x,y)=\sum_{j,k\in\mathbb{Z}}\sum_{n=1}^{r^{(s)}_{jk}} \sigma^{(s)}_{jk,n}\kl{(v^{(s)}_{jk,n})^H\kl{ \spr{ \phi_\ell,w^{(s)}_{jk,\ell} }_{H^s(c_\ell\Omega_T,\gamma_\ell)}}_{\ell=1}^L}u^{(s)}_{jk,n}w_{jk}(x,y)\,,
    \end{equation*}
which can be rewritten as
    \begin{equation*}
        (\tilde{A}^{(s)}\phi)(x,y)=\sum_{j,k\in\mathbb{Z}}\sum_{n=1}^{r^{(s)}_{jk}} \sigma^{(s)}_{jk,n}\spr{ \phi, \kl{ (v^{(s)}_{jk,n})_\ell \, w^{(s)}_{jk,\ell} }_{\ell=1}^L} _{\D(\Tilde{A}^{(s)})}u^{(s)}_{jk,n}w_{jk}(x,y)\,.
    \end{equation*}
This implies the following nullspace and range characterizations
    \begin{equation}\label{eq:aopproof}
    \begin{split}
        \mathcal{N}(\tilde{A}^{(s)})^\perp &= \mathrm{span}\Kl{\left( (v^{(s)}_{jk,n})_\ell \, w^{(s)}_{jk,\ell}\right)_{\ell=1}^L \, \big\vert \, j,k\in\mathbb{Z} \,, 1\leq n \leq r^{(s)}_{jk}} \,,
        \\
        \overline{\mathcal{R}(\tilde{A}^{(s)})} &=\mathrm{span}\Kl{u^{(s)}_{jk,n}w_{jk} \, \big\vert \, j,k\in\mathbb{Z} \,, 1\leq n\leq r^{(s)}_{jk}} \,.
    \end{split}
    \end{equation}
and thus a candidate for a singular system of $\Tilde{A}^{(s)}$ is given by
    \begin{equation*}
        \kl{ \sigma^{(s)}_{jk,n},\left( (v^{(s)}_{jk,n})_\ell \, w^{(s)}_{jk,\ell}\right)_{\ell=1}^L,u^{(s)}_{jk,n}w_{jk}} \,.
    \end{equation*}
To show that this is indeed a singular system, we need to show three properties:
\renewcommand{\labelenumi}{(\roman{enumi})}
\begin{enumerate}
    \item $\left\{\left( (v^{(s)}_{jk,n})_\ell \, w^{(s)}_{jk,\ell}\right)_{\ell=1}^L \, \big\vert\, j,k\in\mathbb{Z}\,, 1\leq n \leq r^{(s)}_{jk}\right\}$ is an orthonormal system,
    \item $\left\{u^{(s)}_{jk,n}w_{jk}\,\big\vert\, j,k\in\mathbb{Z} \,, 1\leq n\leq r^{(s)}_{jk}\right\}$ is an orthonormal system,
    \item for all $j,k\in\mathbb{Z}$ and $1\leq n \leq r^{(s)}_{jk}$ there holds 
        \begin{equation*}
            \tilde{A}^{(s)}\left( (v^{(s)}_{jk,n})_\ell \, w^{(s)}_{jk,\ell}\right)_{\ell=1}^L=\sigma^{(s)}_{jk,n}u^{(s)}_{jk,n}w_{jk} \,.
        \end{equation*}
\end{enumerate}
For this, let $j,j',k,k'\in\mathbb{Z}$ and $1\leq n \leq r^{(s)}_{jk}$, $1\leq n' \leq r^{(s)}_{j'k'}$ be arbitrary but fixed. Using the orthonormality of the functions $w^{(s)}_{jk,\ell}$ on $H^s(c_\ell \Omega_T,\gamma_\ell)$ we obtain
    \begin{equation*}
    \begin{split}
        &\sum_{\ell=1}^L\left\langle\, (v^{(s)}_{jk,n})_\ell\, w^{(s)}_{jk,\ell},(v^{(s)}_{j'k',n'})_\ell \, w^{(s)}_{j'k',\ell} \,\right\rangle_{H^s(c_\ell\Omega_T,\gamma_\ell)}
        \\
        & \qquad =
        \sum_{\ell=1}^L(v^{(s)}_{jk,n})_\ell  \,\overline{(v^{(s)}_{j'k',n'})}_\ell\left\langle\, w^{(s)}_{jk,\ell}, w^{(s)}_{j'k',\ell} \,\right\rangle_{H^s(c_\ell \Omega_T,\gamma_\ell)}
        \\
        & \qquad = \left(v^{(s)}_{j'k',n'}\right)^Hv^{(s)}_{jk,n} \, \delta_{jk,j'k'} = \delta_{n,n'} \,\delta_{jk,j'k'} \,,
    \end{split}    
    \end{equation*}
which establishes (i). Property~(ii) can be shown analogously, and property~(iii) follows directly from \eqref{eq:aopproof} using the orthonormality established in (i). Hence, we can apply the same arguments as in \cite[Theorem 2.8]{Engl_Hanke_Neubauer_1996}, which then yield the assertion.
\end{proof}

As mentioned before, the above analysis generalizes results of \cite{Neubauer_Ramlau_2017,Hubmer_Ramlau_2020}, which were in particular derived for the case $s=0$. In the case of only NGS, i.e., when $c_\ell=1$  for all $\ell=1,\ldots,L$, the corollary below provides an explicit relation between the general case $s > 0$ and this base setting. 

\begin{corollary}
Let $s\geq 0$, let the operators $\Tilde{A}^{(s)}, \Tilde{A}^{(0)}$ be defined as in \eqref{defAperonly}, and let $c_\ell=1$ for all $\ell=1,\ldots,L$ (pure NGS case). Furthermore, let $(\sigma^{(s)}_{jk,n},u^{(s)}_{jk,n},v^{(s)}_{jk,n})$ and $(\sigma^{(0)}_{jk,n}, u^{(0)}_{jk,n},v^{(0)}_{jk,n})$ denote the singular systems of the matrices $\Tilde{A}^{(s)}_{jk}$ and $\Tilde{A}^{(0)}_{jk}$ defined in \eqref{Ajkdef}, respectively. Then $\beta_T:=\beta_{\ell,T}$ is independent of $\ell$ and it follows that
    \begin{equation*}
        (\tilde{A}^{(s)}\phi)(x,y)=\sum_{j,k\in\mathbb{Z}} \left(1+\beta_{T}|(j,k)|^2\right)^{-s}\left(\sum_{n=1}^{r^{(0)}_{jk}}\sigma^{(0)}_{jk,n}\kl{(v^{(s)}_{jk,n})^H\phi_{jk}}u^{(0)}_{jk,n}\right)w_{jk}(x,y)\,,
    \end{equation*}
as well as
    \begin{equation*}
        ({(\Tilde{A}^{(s)})}^\dagger \varphi)_\ell(x,y)=\sum_{j,k\in\mathbb{Z}}\left(1+\beta_{T}|(j,k)|^2\right)^{s}\left(\sum_{n=1}^{r^{(0)}_{jk}}\frac{\kl{(u^{(0)}_{jk,n})^H\varphi_{jk}}}{\sigma^{(0)}_{jk,n}}v^{(0)}_{jk,n}\right)_\ell w_{jk,\ell}(x,y) \,.
    \end{equation*}
\end{corollary}
\begin{proof}
Note that due to \eqref{vjkl} and \eqref{Ajkdef} there holds
    \begin{equation*}
        \Tilde{A}^{(s)}_{jk}=\left(1+\beta_{T}|(j,k)|^2\right)^{-s/2}\Tilde{A}^{(0)}_{jk} \,,
    \end{equation*}
which implies that
    \begin{equation*}
    \begin{aligned}
        (\Tilde{A}^{(s)}_{jk})^H\Tilde{A}^{(s)}_{jk}v^{(0)}_{jk,n}&=\left(1+\beta_{T}|(j,k)|^2\right)^{-s}(\Tilde{A}^{(0)}_{jk})^H\Tilde{A}^{(0)}_{jk}v^{(0)}_{jk,n}\\&=\left(1+\beta_{T}|(j,k)|^2\right)^{-s}(\sigma^{(0)}_{jk,n})^2v^{(0)}_{jk,n} \,, 
    \end{aligned}
    \end{equation*}
and thus $\sigma^{(s)}_{jk,n}=\left(1+\beta_{T}|(j,k)|^2\right)^{-s/2}\sigma^{(0)}_{jk,n}$, $r_{jk}^{(s)}=r_{jk}^{(0)}$ and $v^{(s)}_{jk,n}=v^{(0)}_{jk,n}$. Similarly, we obtain $u^{(s)}_{jk,n}=u^{(0)}_{jk,n}$ which together with \eqref{eq_As_dec} and \eqref{Aperdagger} yields the assertion.
\end{proof}

Note that the above corollary not only connects the SVTDs of $\Tilde{A}^{(s)}$ in the pure NGS-case for different values of $s$, but also sheds some additional light on the ill-posedness of the problem. In particular, due to the factors $\left(1+\beta_{T}|(j,k)|^2\right)^{-s}$, which essentially correspond to the singular values of the Sobolev embedding operator (cf.~\cite{Hubmer_Sherina_Ramlau_2023}), the inversion of $\Tilde{A}^{(s)}$ is ill-posed for each $s > 0$. Furthermore, if \eqref{cond_wellposed} holds, then the degree of ill-posedness is exactly $s/2$, i.e., the same as for the problem of inverting the Sobolev embedding operator. However, if \eqref{cond_wellposed} is not satisfied, then the degree of ill-posedness may be significantly worse.
Note that due to a result by Kolmogorov \cite{Kolmogorov}, a typical atmospheric turbulence layer is expected to satisfy $s = 11/6$. Note that this assumption is commonly used both for the simulation of atmospheric turbulence profiles \cite{Octopus,octopus06} as well as in modern reconstruction algorithms for atmospheric tomography; see e.g.~\cite{Auzinger_2017,Eslitzbichler2013,Yudytskiy_Helin_Ramlau_2014}.

% Subsection - Computational Aspects
\subsection{Computational Aspects}

In this section, we discuss some computational aspects which are relevant for the successful practical application of the SVDs derived above. First, recall that the atmospheric tomography problem is generally ill-posed problem, and thus some form of regularization is required in case of noisy data $\varphi^\delta$, which is always to be expected in practice. For this, one can employ a regularizing filter $g_\alpha:\R\to\R$ (cf.~\cite{Engl_Hanke_Neubauer_1996}) in the generalized inversion formula \eqref{Aperdagger} to obtain the regularized solution
    \begin{equation}\label{SVD_reg}
        \phi_\alpha^\delta :=\sum_{j,k\in\mathbb{Z}}\left(\sum_{n=1}^{r^{(s)}_{jk}}g_\alpha\left(\sigma^{(s)}_{jk,n}\right)\kl{(u^{(s)}_{jk,n})^H\varphi^\delta_{jk}}v^{(s)}_{jk,n}\right)_\ell w^{(s)}_{jk,\ell}(x,y) \,.
    \end{equation}
In order for this approach to become a regularization method, the filters $g_\alpha$ need to be chosen appropriately. In our numerical experiments below, we use a Tikhonov filter of the form $g^{\text{Tikh}}_\alpha(s) :=s/(s^2+\alpha)$ together with a suitably selected regularization parameter $\alpha$. However, other filter functions and truncation can also be used \cite{Engl_Hanke_Neubauer_1996}. In fact, the practical necessity of discretizing the infinite sums in \eqref{SVD_reg} implies that truncation is always present as an additional regularization in any implementation.

For an efficient implementation of the regularized SVD \eqref{SVD_reg}, we suggest equidistant discretization grids on the square domains $c_\ell \Omega_T$ and $\Omega_T$, respectively. This has the advantage that the coefficients $(\varphi_{jk}^\delta)$ and the outer sum over the indices $j,k$ can then be effectively computed via the two-dimensional FFT and IFFT, respectively. Furthermore, note that for a fixed atmospheric tomography setup, the singular systems of the matrices $A^{(s)}_{jk}$ can be precomputed and stored, which is beneficial for the repeated application of the inversion formula as required in an actual AO setup. The resulting inversion algorithm is summarized in pseudo-code in Algorithm~\ref{alg:SVTD}.

\begin{algorithm}
\caption{SVTD}\label{alg:SVTD}
\begin{algorithmic}[1]
\Require $\varphi^\delta_1,\ldots,\varphi^\delta_G\in \R^{N\times N},\, \alpha\in\R^+,\, s\in\R^+$
\For{$g=1,\ldots,G$}
\State $(\varphi^\delta_{jk})_g=\texttt{fft2}(\varphi_g^\delta)$ 
\EndFor
\For{$j,k\in I=\{\lfloor -N/2\rfloor +1,\ldots,\lfloor N/2\rfloor -1\}$} 
\State $[u^{(s)}_{jk,n},\sigma^{(s)}_{jk,n},v^{(s)}_{jk.n}]=\texttt{SVD}(\Tilde{A}^{(s)}_{jk})$
\State $d_{jk,\ell}=\displaystyle\sum_{n=1}^{r^{(s)}_{jk}}\frac{\sigma^{(s)}_{jk,n}}{(\sigma^{(s)}_{jk,n})^2+\alpha}\kl{(u^{(s)}_{jk,n})^H\varphi^\delta_{jk}}v^{(s)}_{jk,n}$
\EndFor
\For{$\ell=1,\ldots,L$}
\State $(\phi_\alpha^\delta)_\ell=\texttt{ifft2}((d_{jk,\ell})_{j,k\in I})$
\EndFor
\State \Return $(\phi_\alpha^\delta)_1,\ldots,(\phi_\alpha^\delta)_L \in\R^{N\times N}$
\end{algorithmic}
\end{algorithm}
\vspace{5pt}

Note that in Algorithm~\ref{alg:SVTD} we have implicitly assumed that both the atmospheric turbulence $\phi$ and the wavefronts $\vphi$ are periodic functions defined on the square domain $\Omega_T$. In practice, wavefronts are typically only given on a subset $\Omega_A \subset \Omega_T$ corresponding to the telescope aperture, and thus need to be extended outside $\Omega_A$. The effect of this extension was studied in \cite{Gerth_2014}, where it was found to produce only very minor errors in the reconstructions along the boundaries of the domains $c_\ell \Omega_T$. Thus, all wavefronts used in the numerical examples of Section~\ref{sec:numerics} are extended by $0.$

% % % % % % % % % % % % % % % % % % % % % % %  
% % % % % % Section - FD  % % % % % % % % % %
% % % % % % % % % % % % % % % % % % % % % % %  
\section{An Explicit Frame Decomposition}\label{sec:fd}

In this section, we return from the periodic atmospheric tomography operator \eqref{defAperonly} to the classic version \eqref{defA}. For this operator, no explicit SVD is currently known, one reason for which is that finding orthonormal basis functions for $L_2(\Omega_\ell,\gamma_\ell)$ which satisfy algebraic properties similar to those of $w_{jk,\ell}$ is difficult. A remedy for this was recently proposed in \cite{Hubmer_Ramlau_2020}, where a frame decomposition (FD) with properties comparable to the SVD was derived. Frames can be seen as generalized bases which do not need to be orthogonal, and thus are generally more flexible. Before discussing the FD of \cite{Hubmer_Ramlau_2020} and our present contribution further, we first need to recall some general background on frames, summarized from \cite{Hubmer_Ramlau_2020} and the seminal works \cite{Christensen_2016,Daubechies_1992}.

\begin{definition}\label{framedef}
A family of functions $\{e_k\}_{k\in K}$ in a Hilbert space $H$ is called a frame over $H$, if and only if there exist constants $0<A\leq B <\infty$ such that      
\begin{equation}\label{eq:framedef}
    A\|f\|_H^2\leq\sum_{k\in K}| \langle\, f,e_k\,\rangle_H|^2\leq B\|f\|_H^2 \,.
\end{equation}
The constants $A$ and $B$ are called frame bounds, and the frame is called tight if $A=B$. Furthermore, for a given frame $\{e_k\}_{k\in K}$ the frame (analysis) operator $F$ and its adjoint (synthesis operator) $F^\ast$ are defined by
\begin{equation*}
    \begin{aligned}
        &F:H\to\ell_2(K)\,,\quad
        &f\mapsto\{\langle\, f,e_k\,\rangle\}_{k\in K}\,,
        \\
        & F^\ast:\ell_2(K)\to H\,,\quad
        &a\mapsto\sum_{k\in K}a_ke_k \,.
    \end{aligned}  
\end{equation*}
\end{definition}

Note that due to \eqref{eq:framedef}, the frame operator satisfies
    \begin{equation*}\label{fopnorm}
        \sqrt{A}\leq\|F\|=\|F^\ast\|\leq \sqrt{B} \,,
    \end{equation*}
and thus the so-called frame operator, defined by 
    \begin{equation}\label{S_op}
        S f:= F^\ast F f=\sum_{k\in K} \skp{f,e_k}e_k\,, 
    \end{equation}
is continuously invertible with $A\, I\leq S\leq B\,I$, which allows the following definition.

\begin{definition} 
Let $\{e_k\}_{k\in K}$ be a frame over the Hilbert space $H$ and define
    \begin{equation}\label{dualdef}
        \widetilde{e}_k:=S^{-1} e_k \,.
    \end{equation}
Then the family of functions $\{\widetilde{e}_k\}_{k\in K}$ is called the \textit{dual frame} of $\{e_k\}_{k\in K}$.
\end{definition}

For an FD of the atmospheric tomography operator, we consider the functions
    \begin{equation}\label{wframes}
        w_{jk,\ell g}(x,y):=w_{jk,\ell}(x,y)I_{\ell g}(x,y)=\frac{\sqrt{\gamma_\ell}}{c_{\ell,g}}w_{jk}((x,y)/c_{\ell,g})I_{\ell g}(x,y) \,,
    \end{equation}
where $I_{\ell g}$ denotes the indicator function of the domain $\Omega_A(\alpha_g h_\ell)$, i.e., 
    \begin{equation}\label{def_Ilg}
        I_{\ell g}(x,y):=I_{\Omega_A(\alpha_g h_\ell)}(x,y) \,,
    \end{equation}
and $T$ in \eqref{def_wjk_wjkl} is such that \eqref{Tbed} holds. The functions $w_{jk,\ell g}$ do in fact form a frame:

\begin{lemma}{\cite[Lemma 4.2, Proposition 4.3]{Hubmer_Ramlau_2020}}
The family of functions $\{w_{jk}\}_{j,k\in\Z}$ forms a tight frame over $L_2(\Omega_A)$ with frame bounds $A=B=1$, and the family of functions $\{w_{jk,\ell,g}\}_{j,k\in\Z,g=1,\ldots,G}$ forms a frame over $L_2(\Omega_\ell,\gamma_\ell)$ with $1\leq A\leq B\leq G$.
\end{lemma}

Following \cite{Hubmer_Ramlau_2020}, an FD for the atmospheric tomography operator is now given by 
    \begin{equation} \label{A_ff}  
        (A \phi)(x,y) =(2T)\sum_{\ell=1}^L \sum_{j,k\in\mathbb{Z}}\kl{\frac{\sqrt{\gamma_\ell}}{c_{\ell,g}}w_{jk}\left(\frac{\alpha_g h_\ell}{c_{\ell,g}}\right)\skp{\phi_\ell,w_{jk,\ell g}}_{L_2(\Omega_\ell,\gamma_\ell)} }_{g=1}^G w_{jk}(x,y) \,.
    \end{equation}
Furthermore, similarly as for the SVD, one may define the operator
    \begin{equation}\label{Aframeinv}
    \begin{split}
        &\mathcal{A}:L_2(\Omega_A)^G\to \prod_{\ell=1}^L L_2(\Omega_\ell,\gamma_\ell)
        \\
        &\qquad \vphi \mapsto (2T)\sum_{j,k\in \mathbb{Z}}\sum_{g=1}^G \kl{ \frac{\sqrt{\gamma_\ell}}{c_{\ell,g} \sigma_g^2}w_{jk}\left(-\frac{\alpha_g h_\ell}{c_{\ell,g}}\right)\spr{\varphi_g,w_{jk}}_{L_2(\Omega_A)^G} \widetilde{w}_{jk,\ell g}}_{\ell=1}^L \,,
    \end{split}
    \end{equation}
where $\sigma_g :=\sqrt{\sum_{\ell=1}^L \gamma_\ell\, c_\ell^{-2}}$. As in \cite{Hubmer_Ramlau_2020}, the operator $\mathcal{A}$ can be shown to be well-defined and bounded, and satisfies the following approximate solution properties.

\begin{theorem}\label{thm_frame_main}
Let the atmospheric tomography operator $A$ be as in \eqref{defA}, let $\mathcal{A}$ be defined as in \eqref{Aframeinv}, and let $\vphi \in L_2(\Omega_A)^G$. Furthermore, assume that the sequences
    \begin{equation}\label{def_al}
        a_\ell := \Kl{ \frac{\sqrt{\gamma_\ell}}{c_{\ell,g} \sigma_g^2}w_{jk}\left(-\frac{\alpha_g h_\ell}{c_{\ell,g}}\right)\spr{\varphi_g,w_{jk}}_{L_2(\Omega_A)^G} }_{jk \in \Z, g=1,\ldots,G} \,,
    \end{equation}
satisfy $a_\ell \in R(F_\ell)$ for all $\ell \in \{1,\ldots,L \}$, where $F_\ell$ denotes the frame operator corresponding to the frame $\{w_{jk,\ell g}\}_{j,k\in\mathbb{Z},g=1,\ldots,G}$. Then $\mathcal{A}\vphi$ is a solution of the equation $A \phi = \vphi$. Furthermore, among all other solutions $\psi \in \mathcal{D}(A)$ of $A\phi = \vphi$ there holds
    \begin{equation*}
        \sum_{g,\ell=1}^{G,L} \sum_{jk \in \Z} \abs{ \spr{(\mathcal{A}\vphi)_\ell,w_{jk,\ell g}}_{L_2(\Omega_\ell,\gamma_\ell)}}^2 
        \leq
        \sum_{g,\ell=1}^{G,L} \sum_{jk \in \Z} \abs{ \spr{\psi_\ell,w_{jk,\ell g}}_{L_2(\Omega_\ell,\gamma_\ell)}}^2 \,.
    \end{equation*}
On the other hand, assume that $A \phi = \vphi$ is solvable, and let $\phi^\dagger$ denote the minimum-norm solution. Then regardless of whether $a_\ell \in R(F_\ell)$ there holds
    \begin{equation*}
        \mathcal{A} \vphi = \phi^\dagger + (\Tilde{F}_\ell^* b_\ell)_{\ell = 1}^L \,, 
        \qquad \text{with} \qquad 
        b_\ell := \Kl{\kl{P_{N(B_{jk})} \phi^\dagger_{jk}) }_{\ell+(g-1)L} }_{jk \in \Z, g = 1,\ldots ,L} \,,
    \end{equation*}
where the vectors $\phi^\dagger_{jk} \in \C^{L\cdot G}$ and the matrices $B_{jk} \in C^{G \times (L \cdot G)}$ are defined by  
    \begin{equation*}
    \begin{split}
        \phi^\dagger_{jk} &:= \operatorname{vec}\Kl{ \kl{\spr{\phi^\dagger, w_{jk,\ell g}}_{L_2(\Omega_\ell,\gamma_\ell)}  }_{\ell=1}^L}_{g=1}^G \,,
        \\
        B_{jk} &:= \operatorname{diag} \Kl{ (2T)\kl{\frac{\sqrt{\gamma_\ell}}{c_{\ell,g}}w_{jk}\left(\frac{\alpha_g h_\ell}{c_{\ell,g}}\right) }_{\ell=1}^L }_{g=1}^G \,.
    \end{split}
    \end{equation*}
\end{theorem}
\begin{proof}
These results were shown for $\gamma_\ell = 1$ in \cite{Hubmer_Ramlau_2020}; see in particular Theorem~4.8 and the subsequent remark. They can be obtained analogously for general $\gamma_\ell$. 
\end{proof}

The definition \eqref{Aframeinv} of the frame inverse $\mathcal{A}$ involves the dual frame functions $\widetilde{w}_{jk,\ell g}$, which had to be computed numerically in \cite{Hubmer_Ramlau_2020,Weissinger2021}. However, we now show that these functions $\widetilde{w}_{jk,\ell g}$ have a simple analytic expression, which in turn leads to an explicit representation of $\mathcal{A}$ that does not involve infinite sums. For this, we denote the overlay functions of the $\ell$-th atmospheric layer as
    \begin{equation}\label{def_Ol}
        O_\ell(x,y):=\sum_{g=1}^G I_{\ell g}(x,y) \,,
    \end{equation}
with $I_{\ell g}$ as in \eqref{def_Ilg}.

\begin{lemma}\label{sop_explicit}
For fixed $\ell$ let $S_\ell := F_\ell^\ast F_\ell$, where $F_\ell$ denotes the frame operator corresponding to the frame $\{w_{jk,\ell g}\}_{j,k\in\mathbb{Z},g=1,\ldots,G}$. Then for all $f\in L_2(\Omega_\ell,\gamma_\ell)$ there holds
    \begin{equation}\label{eq_Sl}
        S_\ell f (x,y)=f(x,y)O_\ell(x,y)\,.
    \end{equation}
\end{lemma}
\begin{proof}
Let $\ell$ and $f \in L_2(\Omega_\ell,\gamma_\ell)$ be arbitrary but fixed and for $(x,y) \in c_{\ell,g} \Omega_T$ define
    \begin{equation}\label{helper_def_h}
        h(x,y) := 
        \begin{cases}
            f(x,y)I_{\ell,g}(x,y) \,, & (x,y) \in \Omega_\ell \,,
            \\
            0 \,, & (x,y) \in (c_{\ell,g} \Omega_T) \setminus \Omega_\ell  \,.
        \end{cases}
    \end{equation}
Since $f \in L_2(\Omega_\ell,\gamma_\ell)$ it follows that $h\in L^2(\clg \Omega_T,\gamma_\ell)$, and thus
    \begin{equation}\label{eq_helper_h}
    \begin{split}
        h(x,y)&=\sum_{j,k\in\Z} \skp{h,w_{jk,\ell}}_{L_2(\clg\Omega_T,\gamma_\ell)} w_{jk,\ell}(x,y)
        \\
        &=\sum_{j,k\in\Z} \skp{f I_{\ell,g},w_{jk,\ell}}_{L_2(\Omega_\ell,\gamma_\ell)} w_{jk,\ell}(x,y) \,.
    \end{split}
    \end{equation}
Hence, using the definition \eqref{S_op} of the operator $S_\ell$ we find that for all $(x,y) \in \Omega_\ell$,
    \begin{equation*}
    \begin{aligned}
        S_\ell f(x,y)
        &=\sumg\sumjk\skp{f,\wjklg}_{L^2(\Omega_\ell,\gamma_\ell)} \wjklg(x,y) \\
        &\overset{\eqref{def_wjk_wjkl}}{=}\sumg \sumjk\skp{f,\wjkl I_{\ell g}}_{L^2(\Omega_\ell,\gamma_\ell)} \wjkl(x,y)I_{\ell g}(x,y) \\
        &=\sumg I_{\ell g}(x,y) \sumjk \skp{f I_{\ell g},\wjk}_{L^2(\Omega_\ell,\gamma_\ell)} \wjkl(x,y) \\
        &\overset{\eqref{eq_helper_h}}{=} \sumg I_{\ell g}(x,y)h(x,y) \overset{\eqref{helper_def_h}}{=} f(x,y)\sumg I_{\ell,g}(x,y)^2 \,,
    \end{aligned}
    \end{equation*}
which together with $I_{\ell,g}(x,y)^2 = I_{\ell,g}(x,y)$ and \eqref{def_Ol} yields the assertion.
\end{proof}

Using this result, we obtain an explicit expression for the functions $w_{jk,\ell g}$.

\begin{corollary}
For fixed $\ell$, let the frame functions $w_{jk,\ell g}$ be defined as in \eqref{wframes}. Then for the dual frame functions $\widetilde{w}_{jk,\ell g}$ of the frame
$\{w_{jk,\ell g}\}_{j,k\in\mathbb{Z},g=1,\ldots,G}$ there holds
    \begin{equation}\label{dualsexplicit}
        \widetilde{w}_{jk,\ell g}(x,y)=w_{jk,\ell g}(x,y)/O_\ell(x,y), \qquad \forall \, (x,y)\in \Omega_\ell \,.
    \end{equation}
\end{corollary}
\begin{proof}
Recall that by definition \eqref{dualdef}, there holds $\widetilde{w}_{jk,\ell g}=S_\ell^{-1}w_{jk,\ell g}$. Hence, using \eqref{eq_Sl} and the fact that $O_\ell$ is nonzero on $\Omega_\ell$ directly yields the assertion.
\end{proof}

For further considerations, we also require an expression for the adjoint of $\mathcal{A}$.

\begin{lemma}
Let the atmospheric tomography operator $A$ be defined as in \eqref{defA} with the FD given in \eqref{A_ff}. Then for the adjoints $A_g^*$ of its components there holds
    \begin{equation}\label{Adjoint_frame}
    \begin{split}
        A_g^\ast:L_2(\Omega_A)&\to \mathcal{D}(A) = \prod_{\ell=1}^L L_2(\Omega_\ell,\gamma_\ell) 
        \\
        \varphi_g & \mapsto (2T)\sum_{j,k\in\mathbb{Z}} \frac{\sqrt{\gamma_\ell}}{c_{\ell,g}}\skp{\varphi_g,w_{jk}}_{L_2(\Omega_A)} w_{jk}\left(-\frac{\alpha_g h_\ell}{c_{\ell,g}}\right)w_{jk,\ell g}
    \end{split}
    \end{equation}
\end{lemma}
\begin{proof}
Using the FD \eqref{A_ff} of the atmospheric tomography operator, we obtain
    \begin{equation*}
    \begin{aligned}
        &\skp{A_g\phi,\varphi_g}_{L_2(\Omega_A)}
        \\
        &=(2T)\sum_{j,k\in\mathbb{Z}}\sum_{\ell=1}^L \frac{\sqrt{\gamma_\ell}}{c_{\ell,g}}w_{jk}\left(\frac{\alpha_g h_\ell}{c_{\ell,g}}\right)\skp{\phi_\ell,w_{jk,\ell g}}_{L_2(\Omega_\ell,\gamma_\ell)}\skp{w_{jk},\varphi_g}_{L_2(\Omega_A)}
        \\
        &=\sum_{\ell=1}^L \skp{\phi_\ell,(2T)\sum_{j,k\in\mathbb{Z}} \frac{\sqrt{\gamma_\ell}}{c_{\ell,g}}w_{jk}\left(-\frac{\alpha_g h_\ell}{c_{\ell,g}}\right)\skp{\varphi_g,w_{jk}}_{L_2(\Omega_A)}w_{jk,\ell g}}_{L_2(\Omega_\ell,\gamma_\ell)}\,,
    \end{aligned}
\end{equation*}
for all $\vphi_g \in L_2(\Omega_A)$ and $\phi \in \mathcal{D}(A)$, which directly yields the assertion.
\end{proof}

Using this representation of the adjoint atmospheric tomography operator, we can now find an explicit representation of the approximate solution operator $\mathcal{A}$.

\begin{theorem}\label{thm_Af_adj}
Let the operator $\mathcal{A}$ be defined as in \eqref{Aframeinv} and let $\varphi\in L_2(\Omega_A)^G$. Then
    \begin{equation}\label{Explicit_frameinv}
    \begin{split}
        (\mathcal{A}\varphi)(r) &= \sum_{g=1}^G\kl{ \frac{\gamma_\ell}{(c_{\ell,g}\sigma_g)^2}\varphi_g\left(\frac{r-\alpha_g h_\ell}{c_{\ell,g}}\right)I_{\Omega_A(\alpha_gh_\ell)}(r)/O_\ell(r) }_{\ell=1}^L
        \\
        &= \sum_{g=1}^G \kl{ \frac{1}{\sigma_g^2 O_\ell(r)} (A_g^* \vphi_g)(r) }_{\ell = 1}^L \,.
    \end{split}
    \end{equation}
\end{theorem}
\begin{proof}
First, note that a comparison of the expressions \eqref{Adjoint} and \eqref{Adjoint_frame} for $A_g^*$ yields
    \begin{equation*}
    \begin{split}
        &\frac{\gamma_\ell}{c_{\ell,g}^2}\varphi_g\left(\frac{r-\alpha_g h_\ell}{c_{\ell,g}}\right)I_{\Omega_A(\alpha_gh_\ell)}(r)
        \overset{\eqref{Adjoint}}{=} A_g^* \vphi_g
        \\
        &\qquad \overset{\eqref{Adjoint_frame}}{=} (2T)\sum_{j,k\in\mathbb{Z}} \frac{\sqrt{\gamma_\ell}}{c_{\ell,g}}\skp{\varphi_g,w_{jk}}_{L_2(\Omega_A)} w_{jk}\left(-\frac{\alpha_g h_\ell}{c_{\ell,g}}\right)w_{jk,\ell g}(r)
        \,.
    \end{split}
    \end{equation*}
Hence, dividing both sides by $\sigma_g^2 O_\ell(r)$ and summing over $g$, we find that 
    \begin{equation*}
    \begin{split}
        &\sum_{g=1}^G\frac{\gamma_\ell}{\sigma_g^2c_{\ell,g}^2}\varphi_g\left(\frac{r-\alpha_g h_\ell}{c_{\ell,g}}\right)I_{\Omega_A(\alpha_gh_\ell)}(r)/O_\ell(r) 
        = \sum_{g=1}^G \frac{1}{\sigma_g^2 O_\ell(r)}(A_g^* \vphi_g)(r)
        \\
        &\qquad=
        (2T)\sum_{j,k\in\mathbb{Z}}\sum_{g=1}^G \frac{\sqrt{\gamma_\ell}}{\sigma_g^2c_{\ell,g}}(\varphi_{jk})_g w_{jk}\left(-\frac{\alpha_g h_\ell}{c_{\ell,g}}\right)w_{jk,\ell g}(r)/O_\ell(r) \,.
    \end{split}
    \end{equation*}
Now, since due to \eqref{dualsexplicit} there holds
$\widetilde {w}_{jk,\ell g} = w_{jk,\ell g}/O_\ell$, the right side of the above expression equals the $\ell$-th component of $\mathcal{A}\vphi$, cf.~\eqref{Aframeinv}, which yields the assertion.
\end{proof}

In Theorem~\ref{thm_frame_main}, we have seen that the frame inverse $\mathcal{A}\vphi$ is a solution of $A \phi = \vphi$ if the sequences $a_\ell$ defined in \eqref{def_al} satisfy $a_\ell \in R(F_\ell)$ for all $\ell \in \{1,\ldots,L\}$. However, this condition is difficult to verify, in particular since the ranges $R(F_\ell)$ are non-trivial subsets of $\ell_2(\mathbb{N})$. But if $a_\ell \notin R(F_\ell)$, then in general $\mathcal{A}\phi$ can only be expected to be an approximate solution in the sense of Theorem~\ref{thm_frame_main}. A potential solution for this would be to project the frames $\{w_{jk}\}_{j,k\in\Z}$ and $\{w_{jk,\ell,g}\}_{j,k\in\Z,g=1,\ldots,G}$ onto $\overline{R(A)}$ and $N(A)^\perp$, respectively \cite{Ebner_Frikel_Lorenz_Schwab_Haltmeier_2023,Hubmer_Ramlau_Weissinger_2022}. However, to do so analytically would essentially require an explicit form of the the Moore-Penrose inverse $A^\dagger$ of $A$, which is not available. Hence, in this paper we propose a remedy in the form of the \emph{iterative FD}, given by
    \begin{equation}\label{iterative_FD}
        \phi_{k+1} = \phi_k + \mathcal{A}(\vphi - A\phi_k) \,.
    \end{equation}
This method essentially amounts to fixed point iteration applied to the modified normal $\mathcal{A}A\phi = \mathcal{A}\vphi$, which we expect to converge to a solution of the atmospheric tomography operator satisfying the range condition $a_\ell \in R(F_\ell)$. While at the moment we cannot provide a full theoretical justification of this approach, the numerical results presented below demonstrate that the iterative FD significantly improves the quality of the obtained reconstructions. This is perhaps not very surprising, given that in Theorem~\ref{thm_Af_adj} we have found that $\mathcal{A}$ corresponds to a specifically
weighted adjoint of $A$, and thus \eqref{iterative_FD} is essentially a gradient method.

In fact, the gradient methods for the atmospheric tomography problem developed in \cite{Saxenhuber_Ramlau_2016} make use of the following adapted inner product on the atmospheric layers:
    \begin{equation*}%\label{weighted_products}
        \skp{\psi,\theta}_\xi = \sum_{\ell=1}^L \frac{1}{\gamma_\ell} \int_{\Omega_\ell}O_\ell(r)  \psi_{\ell}(r) \overline{\theta_{\ell}(r)} \mathrm{d}r \,.
    \end{equation*}
The adjoint of the atmospheric tomography operator with respect to this weighted inner product then takes the form
    \begin{equation*}
        (A_\xi^\ast \vphi)(r) = \sum_{g=1}^G (A_{g,\xi}^\ast \vphi)(r) =
        \sum_{g=1}^G\kl{\frac{\gamma_\ell}{c_{\ell,g}^2}\varphi_g\left(\frac{r-\alpha_g h_\ell}{c_{\ell,g}}\right)I_{\Omega_A(\alpha_gh_\ell)}(r)/O_\ell(r) }_{l=1}^L \,.
    \end{equation*}
Comparing this with \eqref{Aframeinv}, we find that $\mathcal{A}=\sum_{g=1}^G \sigma_g^2 A_{g,\xi}^\ast$, providing another link between the iterative FD \eqref{iterative_FD} and the specific gradient methods derived in \cite{Saxenhuber_Ramlau_2016}.

% % % % % % % % % % % % % % % % 
% % Section - Numerics  % % % %
% % % % % % % % % % % % % % % % 
\section{Numerical Experiments}\label{sec:numerics}

In this section, we give some numerical verification of our developed algorithms. To this end, we consider realistic AO systems motivated by the ELT, which is currently under construction in the Atacama desert in Chile. We first describe the test configurations and simulation environment in detail, and then compare the proposed methods with state of the art algorithms in terms of reconstruction quality.

% Subsection - Test Configuration and Simulation Environment
\subsection{Test Configuration and Simulation Environment}

To evaluate the performance of the proposed algorithms, we use a design similar to that of the ELT, which will become the largest optical/near-infrared telescope in the world. The ELT will be equipped with two so-called Nasmyth platforms on each side containing different instruments. The test setting in this paper is motivated by the instrument MORFEO \cite{MORFEO}, which is an AO module operating in MCAO.

The AO system configuration is shown in Table~\ref{tab:general_setting}. We simulate a telescope which gathers light through a primary mirror with $42$~m diameter and a $28\%$ central obstruction. The ELT optical design consists of three mirrors denoted by M1, M2 and M3 on-axis with two DMs (M4, M5) for performing the AO. For MORFEO, two additional DMs (DM1, DM2) inside the instrument are used for wavefront compensation. Note that we assume the Fried geometry with equidistant actuator spacing for all DMs \cite{Fried_77}. Details on the configuration are listed in Table~\ref{tab:DM}. The turbulence is simulated according to median seeing conditions with a Fried parameter of $0.129$~m. We use the ESO standard 9-layer atmosphere as e.g.\ defined in \cite{Auzinger_2015} and reconstruct $3$ layers at the altitudes of the DMs, see Table~\ref{tab:layers}. Figure~\ref{fig:mirrors_layers} shows a graphical illustration of the DM and layer configuration. The layer heights are chosen such that the light emitted from a certain turbulent layer is exactly focused in the corresponding DM. A wavefront perturbation stemming from one of these layers, which are assumed to be infinitely thin, is then exactly compensated if the corresponding DM is given the correct shape. This special choice of the mirror position along the optical axis is often referred to as conjugation. In this way, we obtain the DM shapes directly from the reconstructed layers and do not have to perform a projection step.

\begin{table}
\footnotesize
\renewcommand{\arraystretch}{1.3}
\centering
\begin{minipage}{.49\textwidth}
	\begin{tabular}{|r|c|}
		\hline
		\textbf{Parameter} & \textbf{Value}\\\hline\hline
		Telescope diameter &  $42$~m\\\hline
		Central obstruction & $28\%$\\\hline
        Fried parameter $r_0$ & $0.129$~m\\\hline
		Na-layer height & $90$~km\\\hline
		Na-layer FWHM & $11.4$~km\\\hline
		Field of View & $2$~arcmin\\\hline
		Evaluation criterion & Strehl ratio\\\hline
		Evaluation wavelength & K band ($2200$~nm) \\\hline
	\end{tabular}
 	\caption{General system parameters.}
  	\label{tab:general_setting}
\end{minipage}
\begin{minipage}{.49\textwidth}
	\begin{tabular}{|r|c|c|c|}
		\hline
		& \textbf{M4} & \textbf{DM1} & \textbf{DM2}\\\hline\hline
		Actuators & $85\times 85$ &  $47\times 47$ & $53\times 53$\\\hline
		Altitude & $0$~km & $4$~km & $12.7$~km\\\hline
		Spacing & $0.5$~m & $1$~m & $1$~m\\\hline
	\end{tabular}
 	\caption{DM configuration.}
   	\label{tab:DM}
\end{minipage}
\end{table}

\begin{figure}[h]
    \centering
    \includegraphics[width=0.7\textwidth]{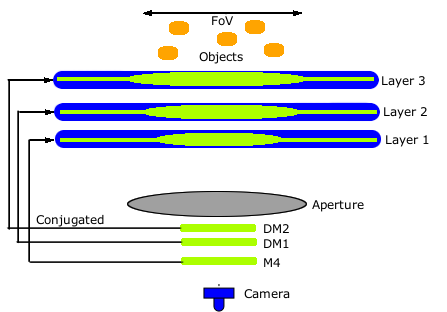}\\
    \caption{Graphical illustration of 3 DMs conjugated to 3 reconstructed layers.}
    \label{fig:mirrors_layers}
\end{figure}

We test our algorithms on two different configurations of natural guide stars (NGS) and laser guide stars (LGS) for measuring the wavefront aberrations, see Figure~\ref{fig:MCAO_asterism}. In the first test case (Figure~\ref{fig:MCAO_asterism}, left) we simulate $6$ bright NGS positioned in a circle of $1$~arcmin diameter. This setting is easier to handle, as the modelling of LGS is more complex due to their finite height. Moreover, wavefronts obtained by LGS face the problem of tip-tilt indetermination \cite{Roddier1999}, i.e., the planar component of the measured wave is unreliable. In practice, it is not realistic to find $6$ bright NGS in the surrounding of the object of interest. Thus, we consider as second test configuration (Figure~\ref{fig:MCAO_asterism}, right) the standard MORFEO setting with $6$ LGS positioned in a circle of $1$~arcmin diameter. We account for the tip-tilt effect via removing the planar component of LGS wavefronts using $3$ faint NGS located in a circle of $160$~arcsec diameter \cite{Gilles_Ellerbroeck_2008}. We model the sodium layer at which the LGS beam is scattered via a Gaussian random variable with mean altitude $H = 90$~km and FWHM of the sodium density profile of $11.4$~km. To each bright NGS and LGS, a Shack-Hartmann (SH) WFS with $84\times84$ subapertures is assigned. The faint NGS for tip-tilt removal are equipped with $2\times2$ SH WFS. The noise induced by the detector read-out is simulated as $3.0$ electrons per pixel and frame. A detector read-out noise (RON) of $3.0$ electrons per pixel per frame is used. For more details see Table~\ref{tab:WFS}.

\begin{table}
\footnotesize
\renewcommand{\arraystretch}{1.3}
\centering
\begin{minipage}{.4\textwidth}
	\begin{tabular}{|c|c|c|}
		\hline
		\textbf{Layer} & \textbf{Altitude} & \textbf{Strength}\\\hline\hline
		$1$ & $0$~m & $0.75$\\\hline
		$2$ & $4000$~m & $0.15$\\\hline
		$3$ & $12700$~m & $0.1$\\\hline
	\end{tabular}
 	\caption{Layer configuration.}
  	\label{tab:layers}
\end{minipage}
\begin{minipage}{.59\textwidth}
	\begin{tabular}{|r|c|c|c|}
		\hline
		& \textbf{LGS-WFS} & \textbf{NGS-WFS} & \textbf{TT-WFS}\\\hline\hline
		Type & SH WFS & SH WFS & SH WFS\\\hline
		Subap. & $84\times 84$ & $84\times 84$ & $2\times 2$\\\hline
		Wavelength & $589$~nm & $589$~nm & $1650$~nm\\\hline
	\end{tabular}
 	\caption{WFS configuration.}
  	\label{tab:WFS}
\end{minipage}
\end{table}

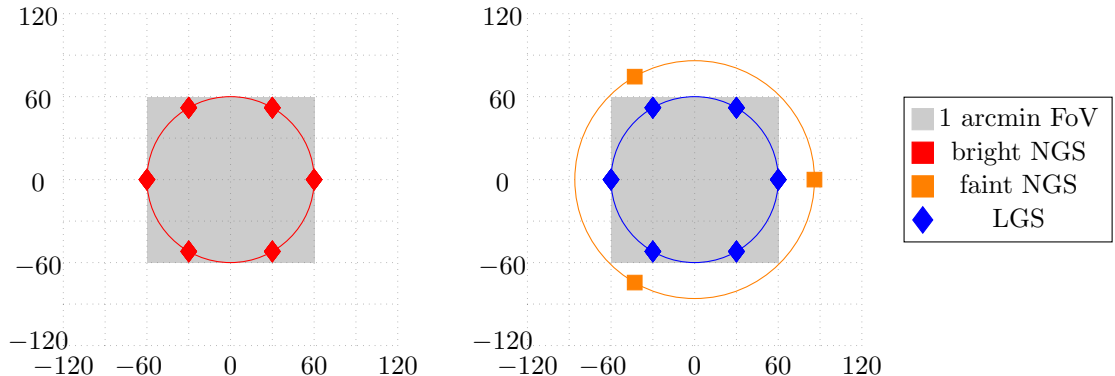
\begin{figure}[ht]
\centering
\begin{tikzpicture}[scale=1.1,font=\footnotesize]
	\fill[gray!40!white] (2,2) rectangle (4,4);
	\draw [very thin, dotted, gray, step=0.5] (0.9999,0.9999) grid (5,5);
	\draw (0.7,1.3) node[below] {$-120$};
	\draw (0.7,2.2) node[below] {$-60$};
	\draw (0.7,3.2) node[below] {$0$};
	\draw (0.7,4.2) node[below] {$60$};
	\draw (0.7,5.2) node[below] {$120$};
	\draw (1,1) node[below] {$-120$};
	\draw (1.9,1) node[below] {$-60$};
	\draw (3,1) node[below] {$0$};
	\draw (4,1) node[below] {$60$};
	\draw (5,1) node[below] {$120$};
	\draw[red] (3,3) circle (1);
	\foreach \x in {2,2.5,3,3.5,4}
	%\draw [red] plot [only marks, mark size=3.5, mark=diamond*] coordinates {(2.4,3.8) (3.6,3.8) (2,3) (4,3) (2.4,2.2) (3.6,2.2)};
    \draw [red] plot [only marks, mark size=3.5, mark=diamond*] coordinates {(2.5,3.866) (3.5,3.866) (2,3) (4,3) (2.5,2.134) (3.5,2.134)};
	\end{tikzpicture}\hspace{0.3cm}
 \begin{tikzpicture}[scale=1.1,font=\footnotesize]
	\fill[gray!40!white] (2,2) rectangle (4,4);
	\draw [very thin, dotted, gray, step=0.5] (0.9999,0.9999) grid (5,5);
	\draw (0.7,1.3) node[below] {$-120$};
	\draw (0.7,2.2) node[below] {$-60$};
	\draw (0.7,3.2) node[below] {$0$};
	\draw (0.7,4.2) node[below] {$60$};
	\draw (0.7,5.2) node[below] {$120$};
	\draw (1,1) node[below] {$-120$};
	\draw (1.9,1) node[below] {$-60$};
	\draw (3,1) node[below] {$0$};
	\draw (4,1) node[below] {$60$};
	\draw (5,1) node[below] {$120$};
	\draw[blue] (3,3) circle (1);
	\draw[orange] (3,3) circle (1.4333);
	\foreach \x in {2,2.5,3,3.5,4}
	%\draw [blue] plot [only marks, mark size=3.5, mark=diamond*] coordinates {(2.4,3.8) (3.6,3.8) (2,3) (4,3) (2.4,2.2) (3.6,2.2)};
    \draw [blue] plot [only marks, mark size=3.5, mark=diamond*] coordinates {(2.5,3.866) (3.5,3.866) (2,3) (4,3) (2.5,2.134) (3.5,2.134)};
	%\draw [orange] plot [only marks, mark size=2.5, mark=square*] coordinates {(2.5,4.333) (2.5,1.633) (4.433,3)};
    \draw [orange] plot [only marks, mark size=2.5, mark=square*] coordinates {(2.283,4.241) (2.283,1.759) (4.433,3)};
	\begin{customlegend}[
	legend entries={
		$1$ arcmin FoV,
		bright NGS,
        faint NGS,
		LGS,
	},
	legend style={at={(8,4)}}]
	\addlegendimage{only marks, mark=square*, color=gray!40!white, mark size=4}
	\addlegendimage{only marks, mark=square*, color=red, mark size=4}
    \addlegendimage{only marks, mark=square*, color=orange, mark size=4}
	\addlegendimage{only marks, mark=diamond*, color=blue, mark size=5}
	\end{customlegend}
	\end{tikzpicture}
	\caption{Different configurations of bright (red) and faint (orange) NGSs and LGSs (blue). The $2$ arcmin FoV is marked in gray.}
	\label{fig:MCAO_asterism}
\end{figure}

We examine the performance of the proposed methods against the Gradient method \cite{Saxenhuber_Ramlau_2016,Saxenhuber_2016} and the Finite Element Wavelet Hybrid Algorithm (FEWHA) \cite{Yudytskiy_2014,Yudytskiy_Helin_Ramlau_2014}, both being iterative reconstruction algorithms for atmospheric tomography. FEWHA uses a dual-domain discretization strategy into wavelet and bilinear basis functions leading to sparse operators. A matrix-free representation of all operators involved makes FEWHA very fast and enables on-the-fly system updates whenever parameters at the telescope or in the atmosphere change \cite{Stadler2020,Stadler2021}. The sparse system is solved using the CG method with a Jacobi preconditoner \cite{Yudytskiy_Helin_Ramlau_2013} and an augmented Krylov subspace method \cite{RaSt2021,Stadler2021} to reduce the number of iterations. The algorithm leads in real-time to an excellent reconstruction quality compared to standard matrix-vector multiplication approaches for the MORFEO instrument \cite{Stadler2021,Stadler2022}. The Gradient method amounts to a steepest descent method applied to a least squares functional, employed with an accelerated step-size, developed and analyzed in \cite{Saxenhuber_Ramlau_2016,Saxenhuber_2016}, here applied with a separated tip-tilt reconstruction in the mixed case. 

In the community of AO, it is common to validate the reconstruction quality using the so called Strehl ratio into certain directions \cite{Roddier1999}. In our simulations, we use $25$ directions positioned in a $5\times 5$ grid over the FoV. The Strehl ratio is defined as the ratio between the maximum of the real energy distribution of incoming light in the image plane $I(x,y)$ over the hypothetical distribution $I_D(x,y)$, which stems from the assumption of diffraction-limited imaging, i.e.,
    \begin{equation*}
        \text{SR}:= \frac{\max_{(x,y)}I(x,y)}{\max_{(x,y)}I_D(x,y)} \,.
    \end{equation*}
By definition, the Strehl ratio is between $0$ and $1$ and frequently given in percent. A Strehl ratio of $1$ means that the influence of the atmosphere has been removed from the observation. For its numerical evaluation the Marechal criterion is used \cite{Roddier1999}.

For our SVTD method summarized in Algorithm~\ref{alg:SVTD}, we used $s=1$ for the Sobolev index, which is slightly smaller than the choice $s=11/6$ corresponding to the expected smoothness of a typical atmosphere \cite{Kolmogorov}. This was done to avoid the commonly observed oversmoothing effect of the adjoint embedding operator \cite{Hubmer_Sherina_Ramlau_2023,Ramlau_Teschke_2004_1}.

All simulations have been carried out in the internal and entirely MATLAB-based AO simulation tool MOST \cite{Auzinger_2017}, which has been developed by the Austrian Adaptive Optics team as an alternative to OCTOPUS \cite{LeLouarn2006,Octopus}, the end-to-end simulator of the European Southern Observatory. The performance of the algorithms is evaluated using the Strehl ratio in the K band, i.e., at a wavelength of $2200$~nm.

% Subsection - Numerical Results
\subsection{Numerical Results}

As a first step of performance evaluation, we consider not a whole end-to-end AO simulation, but compare only the tomographic reconstruction of the proposed methods. We simulate a $3$ layer atmosphere (Figure~\ref{fig:atmosphere}, top row) and perform its reconstruction with the SVTD (Figure~\ref{fig:atmosphere}, middle row) and the iterative FD (Figure~\ref{fig:atmosphere}, bottom row) methods. We observe that both algorithms provide a very good result at the lowest layer $\ell=1$, but fail to reconstruct details at higher altitudes, in particular around the borders. This behaviour is expected, and generally observed for atmospheric tomography, since the reconstruction quality depends on the number of overlapping regions, which decreases with height and with a larger distance from the center, which is often referred to as field off-axis position \cite{Ramlau_Stadler_2024}.

\begin{figure}[h]
    \centering
    \includegraphics[width=0.75\textwidth]{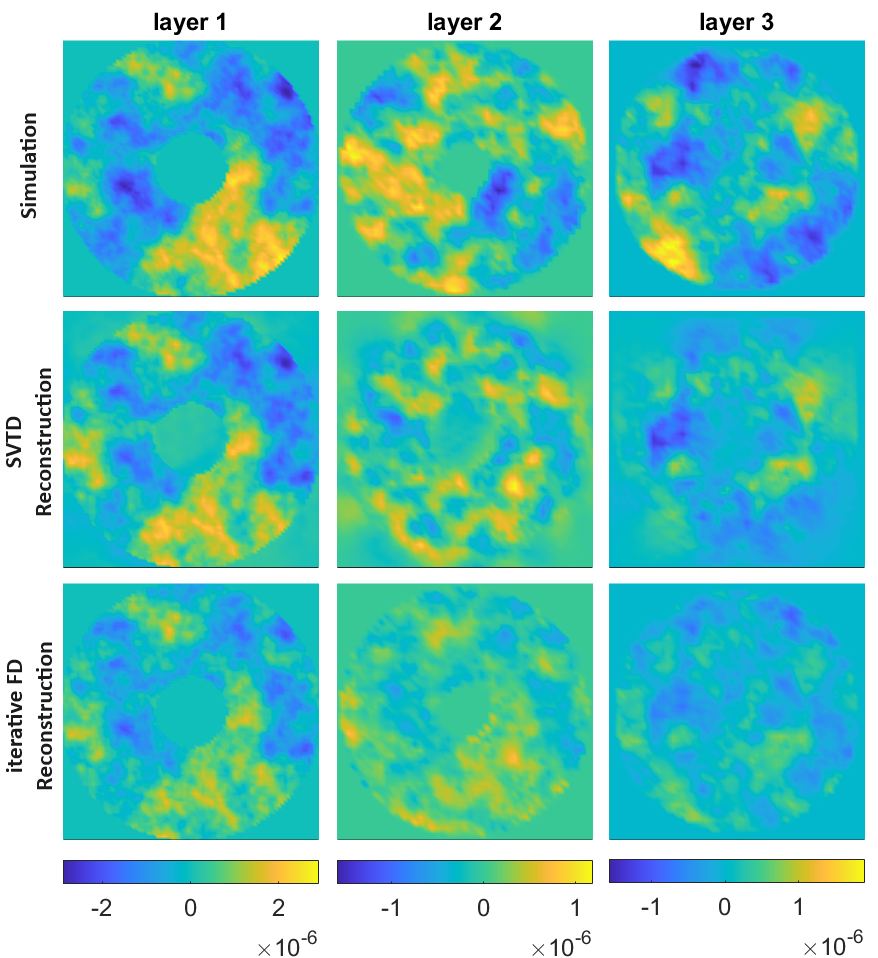}\\
    \caption{A simulated 3 layer atmosphere (top), and its reconstruction with SVTD (middle) and iterative FD (bottom).}
    \label{fig:atmosphere}
\end{figure}

Since all algorithms yield promising tomographic reconstructions, we next consider full end-to-end AO simulations and evaluate the short-exposure (SE) and long-exposure (LE) Strehl ratio. The SE Strehl ratio provides a measure for the reconstruction quality at the current time step, whereas the LE Strehl ratio is an average over time \cite{Roddier1999}. We compare the SVTD and FD approaches with FEWHA and the Gradient method, which are both known to provide very good results for the considered test configurations. For all results with the SVTD method, the regularization parameter $\alpha$ and the smoothing parameter $s$ were tuned experimentally a-priori. The results for the iterative FD and the Gradient method use $5$ iterations, whereas FEWHA uses $4$ iterations in its internal CG method. The method specific parameters for FEWHA have been tuned via simulations.

First, we consider the NGS-only case as illustrated in the left plot of Figure~\ref{fig:MCAO_asterism}. We simulate $1000$ time steps, corresponding to a $2$ seconds interval.  In Figure~\ref{fig:methods1000it_NGS}, we show the average SE Strehl ratio over the $25$ evaluation directions over time (left) and the LE Strehl ratio versus the field off-axis position (right). Note that it is a well known issue for MCAO systems that the quality of the AO correction degrades if we move further away from the center, see e.g. \cite{Ramlau_Stadler_2024}. We observe that all methods except the FD provide very good results, which are stable over time. In particular, the SVTD can compete with the Gradient method and FEWHA, even outperforming them in the center of the FoV. This also becomes evident when looking at the SE Strehl ratio at specific directions depicted Figure~\ref{fig:methods1000it_NGS_directions}. The reason for the suboptimal performance of the FD is likely due to the assumptions on the sequences $a_\ell$ in Theorem~\ref{thm_frame_main} not being satisfied. Consistently with the discussion at the end of Section~\ref{sec:fd}, the iterative FD performs much better than the FD itself, with results comparable to those of the Gradient method. Note that we did not focus on tuning the step-size in the iterative FD approach, which could possibly yield a performance improvement. 

\begin{figure}[h]
    \centering
    % [trim={left bottom right top},clip]
    \includegraphics[width=.49\textwidth,trim={1.1cm 0 1.5cm 0},clip]{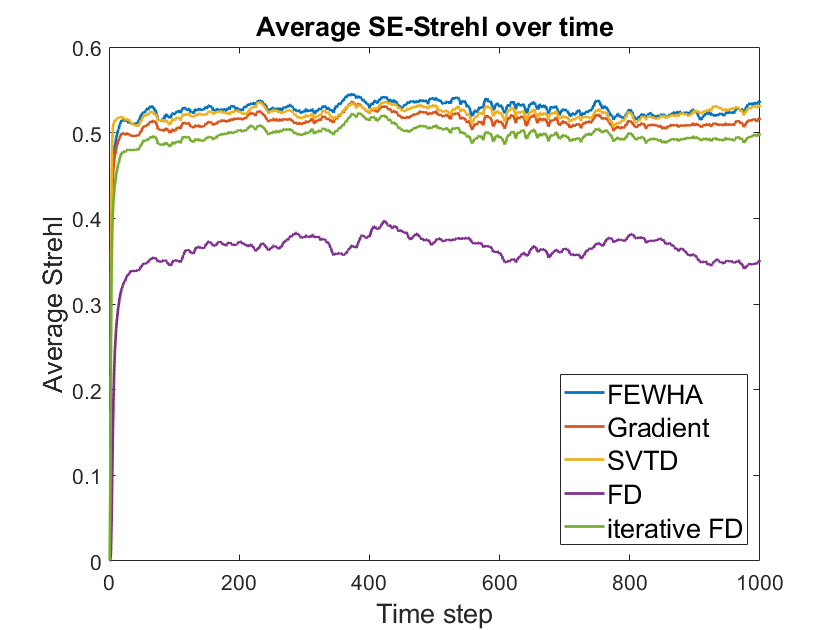}
    \includegraphics[width=.49\textwidth,trim={0.8cm 0 1.7cm 0},clip]{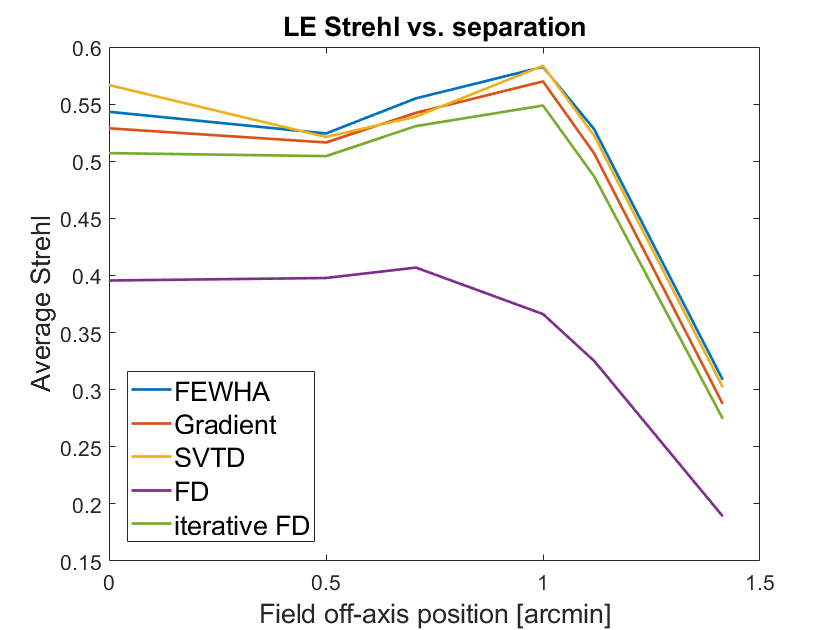}
    \caption{NGS-only case: Average SE Strehl ratio over time (left) and LE Strehl ratio vs. separation (right).}
    \label{fig:methods1000it_NGS}
\end{figure}

\begin{figure}[h]
    \centering
    % [trim={left bottom right top},clip]
    \includegraphics[width=\textwidth,trim={0.8cm 0 1.5cm 0},clip]{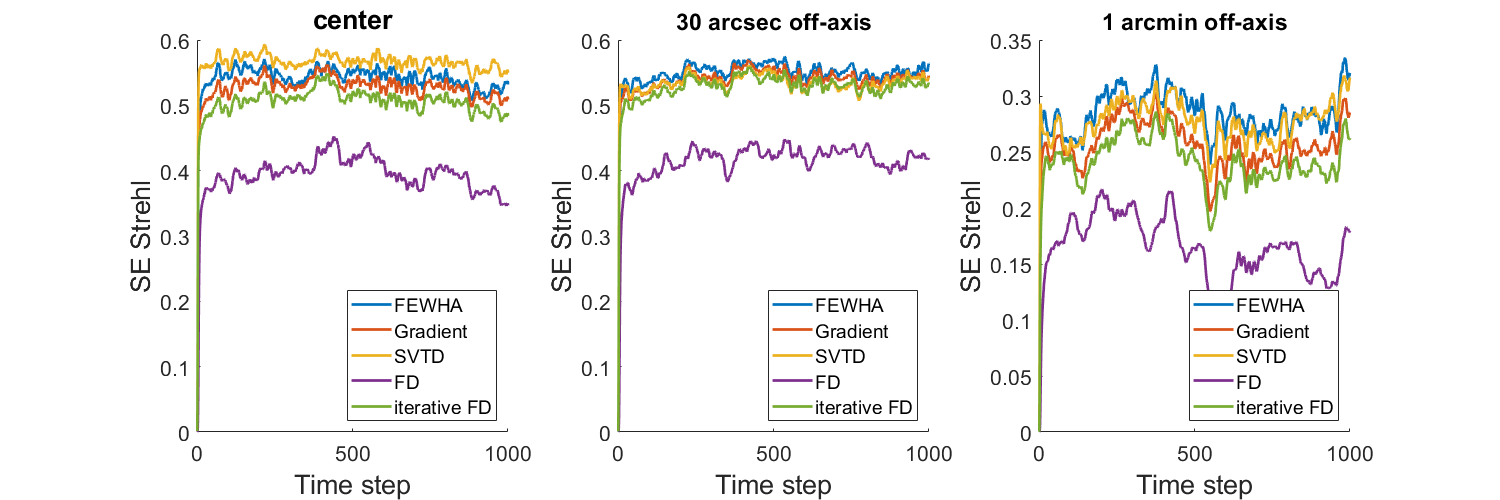}
    \caption{NGS-only case: SE Strehl ratio over time in the center (left), $30$ arcsec off-axis (middle) and $1$ arcmin off-axis (right).}
    \label{fig:methods1000it_NGS_directions}
\end{figure}

Next, we consider the more realistic mixed case in which $6$ LGS are combined with $3$ NGS for tip-tilt removal; cf.\ the right plot of Figure~\ref{fig:MCAO_asterism} for an illustration of the star asterism. We observe from Figure~\ref{fig:methods1000it_mixed} and Figure~\ref{fig:methods1000it_mixed_directions} that FEWHA provides the best result in all directions considered over the whole simulation duration. In general, all algorithms yield more unstable results, i.e., oscillations of the SE Strehl ratio over time, especially for the outer directions. This is expected when moving from an NGS-only setting to a more realistic setting with LGS and faint NGS.  

\begin{figure}[h]
    \centering
    % [trim={left bottom right top},clip]
    \includegraphics[width=.49\textwidth,trim={0.8cm 0 1.5cm 0},clip]{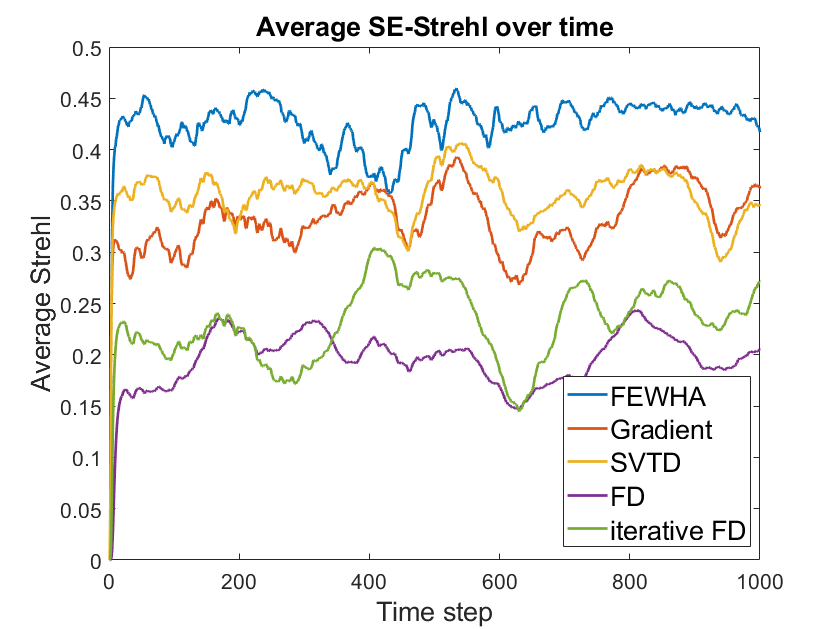}
    \includegraphics[width=.49\textwidth,trim={0.8cm 0 1.7cm 0},clip]{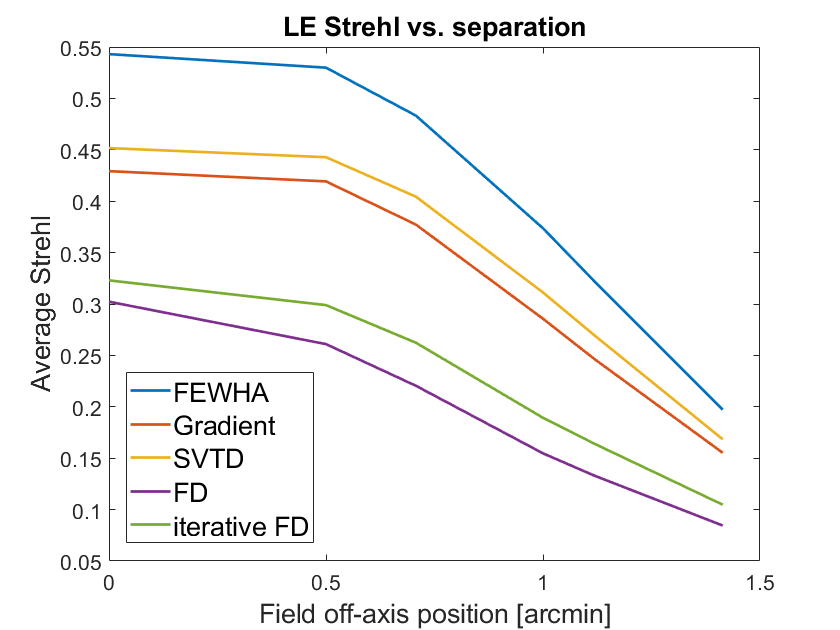}
    \caption{Mixed case: Average SE Strehl ratio over time (left) and LE Strehl ratio vs. separation (right).}
    \label{fig:methods1000it_mixed}
\end{figure}

\begin{figure}[h]
    \centering
    % [trim={left bottom right top},clip]
    \includegraphics[width=\textwidth,trim={0.8cm 0 1.5cm 0},clip]{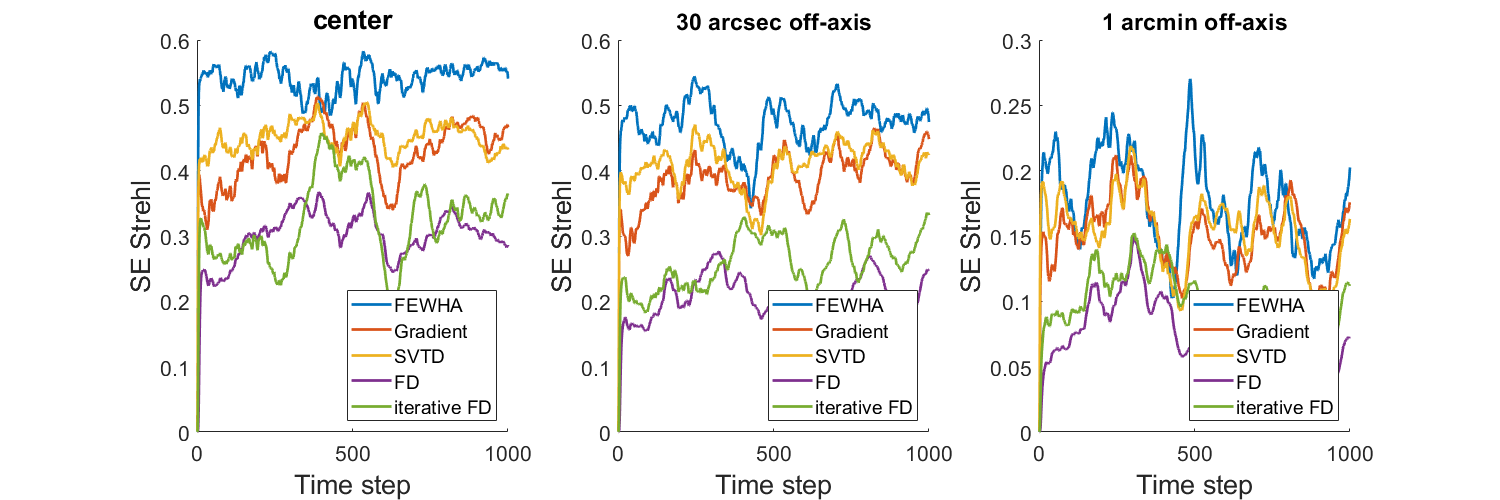}
    \caption{Mixed case: SE Strehl ratio over time in the center (left), $30$ arcsec off-axis (middle) and $1$ arcmin off-axis (right).}
    \label{fig:methods1000it_mixed_directions}
\end{figure}

In terms of computational speed, the MATLAB implementations of all tested methods require a similar run-time, in the order of a few seconds for a single time step. Note that there exists a parallel and matrix-free implementation of FEWHA in C++, which takes only a few milliseconds per time step and thus is able to fulfill the real-time requirements of the ELT \cite{Stadler2020}. The FEWHA MATLAB implementation is matrix-based, non parallel and not optimized for speed, and thus much slower.

% % % % % % % % % % % % %
% Section - Conclusion  %
% % % % % % % % % % % % %
\section{Conclusion} 

In this paper, we considered singular value and frame decompositions of the atmospheric tomography operator. First, we extended existing SVDs for the periodic atmospheric tomography operator to a more realistic Sobolev space setting including weighted inner products which incorporate known or measured turbulence profiles. Then, we considered an FD for the non-periodic atmospheric tomography operator, and derived an explicit representation of the corresponding (approximate) frame inversion operator. Based on these theoretical results, we then developed efficient numerical solution methods for the atmospheric tomography problem, which we implemented and tested in the AO simulation tool MOST. The obtained results, especially those for the SVTD, are very promising, and warrant further investigation. In particular, we plan to implement the SVTD method in non-MATLAB based AO simulation environments, and to leverage its low computational cost and parallelizability to compete with the FEWHA algorithm not only in terms of reconstruction quality as demonstrated above, but also in essential real-time requirements.

% % % % % % % % % % %
% Section - Support %
% % % % % % % % % % %
\section*{Acknowledgement}

The authors thank Andreas Obereder and Stefan Raffetseder for helpful input. This research was funded in part by the Austrian Science Fund (FWF) SFB 10.55776/F68 ``Tomography Across the Scales'', project F6805-N36 (Tomography in Astronomy). For open access purposes, the authors have applied a CC BY public copyright license to any author-accepted manuscript version arising from this submission. Furthermore, the authors were supported by the Austrian Research Promotion Agency (FFG) project number FO999888133, as well as the NVIDIA Corporation Academic Hardware Grant Program. LW is partially supported by the State of Upper Austria. 

\printbibliography

%===============================================================================================

\end{document}